\documentclass[11pt]{article}
\pdfoutput=1

\def\pb{\pagebreak}   

\usepackage{lineno,lmodern} 

\usepackage[dvipsnames]{xcolor}

\def\beq{\begin{equation} }\def\eeq{\end{equation} }\def\1{\mathbf{1}}

\usepackage[framemethod=default]{mdframed}
\usepackage{caption}

\usepackage{indentfirst}
\usepackage{bm, mathrsfs, graphics,float,amssymb,amsmath,subeqnarray,setspace,graphicx,amsthm,epstopdf,subfigure, enumerate, color}
\usepackage[utf8]{inputenc}
\usepackage[colorlinks,
linkcolor=red,
anchorcolor=blue,
citecolor=blue
]{hyperref}
\usepackage{natbib}

\usepackage{fullpage}
\numberwithin{equation}{section}

\newtheorem{lemma}{Lemma}
\newtheorem{theorem}{Theorem}
\newtheorem{proposition}{Proposition}

\ifx\assumption\undefined
\newtheorem{assumption}{Assumption}
\fi

\newcommand{\cO}{\mathcal{O}}

\newcommand{\cN}{\mathcal{N}}

\newcommand{\EE}{\mathbb{E}}
\newcommand{\RR}{\mathbb{R}}
\newcommand{\Ab}{\bm{A}}

\newcommand{\Bb}{\bm{B}}
\newcommand{\xb}{\bm{x}}
\newcommand{\yb}{\bm{y}}

\newcommand{\ub}{\bm{u}}

\newcommand{\I}{\bm{I}}

\newcommand{\PP}{\mathbb{P}}

\newcommand{\argmin}{\mathop{\mathrm{argmin}}}
\newcommand{\argmax}{\mathop{\mathrm{argmax}}}

\usepackage{bbm}

\usepackage{multirow}
\usepackage{tablefootnote}
\usepackage{colortbl}
\usepackage{hhline}

\usepackage{algorithm}
\usepackage{algorithmic}

\begin{document}
\title{
Stochastic Diagonal Estimation Based on Matrix Quadratic Form Oracles
}

\author{
Haishan Ye
\thanks{
Xi'an Jiaotong University;
email: yehaishan@xjtu.edu.cn
}
\and
Xiangyu Chang 
\thanks{
Xi'an Jiaotong University;
email: xiangyuchang@xjtu.edu.cn
}
}
\date{\today}

\maketitle

\def\TH{\tilde{H}}
\newcommand{\ti}[1]{\tilde{#1}}
\def\diag{\mathrm{diag}}
\newcommand{\norm}[1]{\left\|#1\right\|}
\newcommand{\dotprod}[1]{\left\langle #1\right\rangle}
\def\tr{\mathrm{tr}}

\begin{abstract}
We study the problem of estimating the diagonal of an implicitly
given matrix $\Ab$. For such a matrix we have access to an oracle that allows
us to evaluate the matrix quadratic form $ \ub^\top \Ab \ub$. 
Based on this query oracle, we propose a stochastic diagonal estimation method with random variable $\ub$ drawn from the standard Gaussian distribution.
We provide the element-wise and norm-wise sample complexities of the proposed method.
Our numerical experiments on different types and dimensions matrices demonstrate the effectiveness of our method and validate the tightness of theoretical results.
	
\end{abstract}

\section{Introduction}

Estimating the diagonal entries of a matrix is important in many areas of science and engineering.
Originally, the diagonal estimation is used in electronic structure calculations \citep{bekas2007estimator,goedecker1995tight,goedecker1999linear}. 
Because the diagonal preconditioners can accelerate the convergence of iterative optimization algorithms \citep{wathen2015preconditioning}, the diagonal estimator is applied in accelerating training  machine learning models \citep{yao2021adahessian,liu2023sophia}.  
The diagonal estimation is also used in network science \citep{constantine2017global,kucherenko2015derivative}.

Because of its importance, the diagonal estimation has attracted much research attention in the past years \citep{hallman2023monte,baston2022stochastic,bekas2007estimator,epperly2024xtrace}. 
These diagonal estimation algorithms construct the following matrix 
\begin{equation}\label{eq:AA}
	\widehat{\Ab} = \frac{1}{N} \sum_{i=1}^{N} \Ab \bm{w}_i\bm{w}_i^\top,
\end{equation} 
where $\Ab\in\RR^{d\times d}$ and $\bm{w}_i \in\RR^d$ are independent random vectors,
and use the diagonal entries of $\widehat{\Ab}$ as the diagonal estimation of $\Ab$.
Thus, this kind of diagonal estimation is mainly based on the matrix-vector product.

In this paper, we will study stochastic diagonal estimation by only accessing the quadratic form of a matrix $\ub^\top \Ab \ub$ and we will construct the following estimation:
\begin{equation}\label{eq:g0}
	\bm{g} = \frac{1}{2N} \sum_{j=1}^{N} \left[\ub^{(j)}\right]^\top \Ab \ub^{(j)}\cdot \Big([\ub^{(j)}]^2 - \bm{1}_d\Big),
\end{equation}
where $\ub^{(j)}$'s are independent standard Gaussian vectors, $\mathbf{1}_d$ is a $d$ dimensional all one vector, and $[\ub]^2$ denotes the element-wise square.
Compared with the diagonal estimation based on Eq.~\eqref{eq:AA} which needs the matrix-vector product, Eq.~\eqref{eq:g0} can only access the quadratic form of a matrix $\ub^\top \Ab \ub$. 

Our study is motivated by the zeroth-order optimization for high-dimensional problems. 
Recently, post-training large language models  with first-order algorithms suffers from the giant memory costs.
Thus the memory-efficient zeroth-order algorithms attract much research attention because they only need forward passes and achieve the memory efficiency \citep{malladi2023fine,zhao2024second,zhang2024revisiting,dery2024everybody,Chen2025}.  
Similar to the first-order algorithms, one can also use the Hessian diagonal as the  preconditioner to accelerate the convergence of zeroth-order algorithms \citep{zhao2024second}. 
Because zeroth-order algorithms can only access function values, the diagonal estimation in Eq.~\eqref{eq:AA} which requires Hessian-vector product can \emph{not} be used to estimate the Hessian diagonals. 
Instead, our estimation in Eq.~\eqref{eq:g0} only needs the quadratic form $\ub^\top \Ab \ub$ with $\Ab$ being the Hessian $\nabla^2 f(\xb)$ which can be approximately computed by three function value queries as follows: 
\begin{equation}\label{eq:q_form}
	\ub^\top \Ab \ub \approx \frac{f(\xb + \alpha \ub) + f(\xb - \alpha \ub) - 2f(\xb)}{\alpha^2},
\end{equation}
where $0<\alpha$ is a constant of  a small value. Please refer to Proposition~\ref{prop:uAu} to get a detailed description.


\subsection{Literature Review}

To our knowledge, Monte Carlo diagonal estimators based on matrix-vector product were first
studied by  \citet{bekas2007estimator}. 
This kind of diagonal estimation is further explored in \citep{hallman2023monte,baston2022stochastic,epperly2024xtrace}.
In contrast, our work studies the diagonal estimation which only accesses the quadratic form of a matrix instead of accessing the matrix-vector product.

Another important research topic closely related to our study is the Hessian approximation by zeroth-order oracles and 
several works have been proposed  \citep{ye2023mirror,Lyu,lattimore2023second}.  
More  Hessian matrix estimation methods by function values can be found in Chapter~6 of \citet{prashanth2025gradient}. 
Our diagonal estimation is a special case of the Hessian approximation when matrix $\Ab$ is the Hessian.
However, the sample complexities of Hessian approximation are rarely analyzed \citep{Lyu,lattimore2023second,prashanth2025gradient}.
Though \citet{ye2023mirror} provides a sample complexity of Hessian approximation, the result of \citet{ye2023mirror} can \emph{not} derive the sample complexities in this work and our work achieves much tighter sample complexities.

\subsection{Contributions}

The novel features of our contributions are the following:
\begin{enumerate}
	\item We propose a stochastic diagonal estimation method that only queries matrix quadratic form oracles. 
	Our method does \emph{not} require the matrix to be symmetric nor positive (semi)-definite.
	\item We provide the sample complexity analysis of our method to achieve the target precision. 
	Our element-wise bound (Theorem~\ref{thm:main}) shows that the sample complexity mainly depends on the square of the trace and the square of the ``Frobenius'' norm of $\Ab + \Ab^\top$. 
	Thus, to achieve the same relative error, diagonal entries with small absolute values require large sample sizes.
	Our numerical experiments (Section~\ref{sec:exp}) confirms this in different types and different dimension matrices.
	\item We provide the sample complexity of norm-wise bound (Theorem~\ref{thm:main1}) and show that it linearly depends the dimension $d$. 
	The trace of the matrix and  off-diagonal entries of $\Ab + \Ab^\top$ mainly determine the sample complexity.  
	\item Our numerical studies show the effectiveness of our algorithm and validate the tightness of our theoretical result about the sample complexities.
\end{enumerate}


%


\subsection{Notation and Preliminaries}

\paragraph{Notation} In this paper, given a $d\times d$ matrix $\Ab$, we use $\Ab_{i,j}$ with integers $1\leq i,j \leq d$ to denote the entry of $\Ab$ in the $i$-th row and $j$-th column. 
Furthermore, we will follow the Matlab convention and we define $\diag(\Ab)  = [\Ab_{1,1}, \dots, \Ab_{d,d}]^\top$.
We use  $\Ab_{i,:}$ and $\Ab_{:, j}$ to denote the $i$-th row and $j$-th column of $\Ab$. 
We also define the trace of matrix $\tr(\Ab) = \sum_{i=1}^d \Ab_{i,i}$ and the ``Frobenius'' norm $\norm{\Ab} = \sqrt{ \tr(\Ab^\top \Ab) }$.

We denote by $\PP\{E\}$ the probability of an event $E$ and by $\EE[Z]$ the expectation of
a random variable $Z$, which can be a scalar-, vector-, or matrix-valued.

\paragraph{Zeroth-order Optimization}

In the zeroth-order optimization, one can only access the function value $f(\xb)$ instead of the gradient $\nabla f(\xb)$. 
Next, we will give a proposition that describes the relation between zeroth-order optimization and matrix quadratic form.
\begin{proposition}\label{prop:uAu}
Let function $f(\xb)$ be of twice continuous with $\xb\in\RR^d$ and its Hessian be $\gamma$-Lipschitz continuous. 
Given a random Gaussian vector $\ub \in\RR^d$, then it holds that
\begin{equation*}
\ub^\top \nabla^2 f(\xb) \ub 
=
\frac{f(\xb + \alpha \ub) + f(\xb - \alpha \ub) - 2f(\xb)}{\alpha^2} -   \frac{\Delta(\xb, \alpha\ub) + \Delta(\xb, -\alpha\ub)}{\alpha^2},
\end{equation*} 
where $0<\alpha$ is a constant of a small value and $\Delta(\xb, \alpha\ub)$ is defined as 
\begin{align*}
\Delta(\xb, \alpha\ub) = f(\xb + \alpha \ub) - \Big( f(\xb) + \alpha \nabla^\top f(\xb) \ub + \frac{\alpha^2}{2} \ub^\top \nabla^2 f(\xb) \ub \Big).
\end{align*}
Furthermore, with a probability $1-\delta$ with $0<\delta<1$, it holds that
\begin{align*}
\left| \frac{\Delta(\xb, \alpha\ub) + \Delta(\xb, -\alpha\ub)}{\alpha^2} \right| \leq \frac{\alpha \left( 2d + 3 \log\frac{1}{\delta} \right)^3}{3}.
\end{align*}
\end{proposition}
Proposition~\ref{prop:uAu} shows that given a function $f(\xb)$, one can obtain the quadratic form of the Hessian by three function values but with some small perturbation.

\section{Stochastic Diagonal Estimation Based on Quadratic Form Oracles}

\subsection{Algorithm Description}
Our stochastic diagonal estimation method lies on the following proposition. 
\begin{lemma}
Given a matrix $\Ab\in\RR^{d\times d}$ and a random Gaussian vector $\ub\sim \cN(\bm{0}, \I_d)$, and an integer $1\leq p\leq d$, then it holds that
\begin{equation}\label{eq:E}
\frac{1}{2}\cdot\EE\left[\ub^\top \Ab \ub\cdot \bm{u}_p^2 - \ub^\top \Ab \ub\right]
=  \Ab_{p,p}.
\end{equation}
\end{lemma}
\begin{proof}
First, we have
\begin{align*}
\EE\left[\ub^\top \Ab \ub\cdot \bm{u}_p^2\right]
=&
\EE\left[ \sum_{i, j=1}^{d} \Ab_{i,j}\ub_i \ub_j \cdot \ub_p^2   \right]
=
\EE\left[ \sum_{i=1}^{d} \Ab_{i,i} \ub_i^2 \cdot \ub_p^2 \right]\\
=& \EE\left[ \sum_{i=1, i\neq p}^{d} \Ab_{i,i} \ub_i^2 \cdot \ub_p^2 + \Ab_{p,p} \ub_p^4 \right] \\
\stackrel{\eqref{eq:Eu}}{=}&
\sum_{i=1,i\neq p}^{d} \Ab_{i,i} + 3 \Ab_{p,p}
= \tr(\Ab) + 2\Ab_{p,p}.
\end{align*}

Furthermore, 
\begin{align*}
\EE\left[ \ub^\top \Ab \ub\right]
=
\EE\left[\tr(\ub^\top \Ab \ub)\right]
=
\EE\left[ \tr(\Ab \ub\ub^\top) \right]
= \tr(\Ab).
\end{align*}

Thus, it holds that
\begin{align*}
\EE\left[\ub^\top \Ab \ub\cdot \bm{u}_p^2 - \ub^\top \Ab \ub\right] = 2 \Ab_{p,p}.
\end{align*}
\end{proof}

Eq.~\eqref{eq:E} provide an unbiased estimation of diagonal entries of $\Ab$. 
Then, we can take $N$ samples to reduce the variance of stochastic estimation to obtain a high precision estimation.
Letting random variables $\ub^{(j)} \sim\cN(\bm{0}, \I_d)$ with $j = 1,\dots, N$, then we construct the following diagonal estimation 
\begin{equation}\label{eq:g}
\bm{g} = \frac{1}{2N} \sum_{j=1}^{N} \left[\ub^{(j)}\right]^\top \Ab \ub^{(j)}\cdot \Big([\ub^{(j)}]^2 - \bm{1}_d\Big). 
\end{equation}
It is easy to check that $\bm{g}$ in Eq.~\eqref{eq:g} is also an unbiased estimation of diagonal entries of $\Ab$ but with a variance depending on $1/N$.  
The detailed algorithm description is listed in Algorithm~\ref{alg:SA}.

\begin{algorithm}[t]
\caption{Stochastic Diagonal Entries Estimation  Based on Quadratic Form Oracles}
\label{alg:SA}
\begin{small}
\begin{algorithmic}[1]
\STATE {\bf Input:} A query $Q_{\Ab}(\cdot)$ which returns $Q_{\Ab}(\ub) = \ub^\top \Ab \ub$, sample size $N$
\STATE Generate $N$ random Gaussian vectors $\ub^{(j)} \sim \cN(0, \I_d)$
\STATE Obtain $[\ub^{(j)}]^\top \Ab \ub^{(j)}$ by accessing $Q_{\Ab}(\ub^{(j)})$
\STATE {\bf Output:} $\frac{1}{2N} \sum_{j=1}^{N} \left[\ub^{(j)}\right]^\top \Ab \ub^{(j)}\cdot \Big([\ub^{(j)}]^2 - \bm{1}_d\Big)  $
\end{algorithmic}
\end{small}
\end{algorithm}

\subsection{Sample Complexity}

First, we give a sample complexity of Algorithm~\ref{alg:SA} to achieve an element-wise error bound.
\begin{theorem}\label{thm:main}
Let $\bm{g}$ defined in Eq.~\eqref{eq:g} be the output of Algorithm~\ref{alg:SA}. 
Given parameters $0<\delta<1$ and $0<\varepsilon$, then with a probability at least $1-\delta$, it holds that
\begin{equation}\label{eq:eps}
\Ab_{p,p} - \varepsilon \leq \bm{g}_p \leq \Ab_{p,p} + \varepsilon,
\end{equation}
if $N$ satisfies that
\begin{equation}\label{eq:N}
N \geq \frac{V}{4\delta \varepsilon^2}, 
\quad\mbox{ with }\quad V = 2 \Big(\tr(\Ab) + 4 \Ab_{p,p}\Big)^2 
+ \norm{\Ab + \Ab^\top}^2 
+ 8\norm{\Ab_{p,:}^\top + \Ab_{:,p}}^2
- 12 \Ab_{p,p}^2.
\end{equation}
\end{theorem}

Theorem~\ref{thm:main} shows that the sample complexity of Algorithm~\ref{alg:SA} depends on $\varepsilon^{-2}$ which is the same as the ones of stochastic diagonal entries estimation with queries to matrix-vector product \citep{baston2022stochastic,hallman2023monte,cortinovis2022randomized}.
Furthermore, the sample complexity of Algorithm~\ref{alg:SA} to estimating $\Ab_{p,p}$ depends on $\Ab_{p,p}$, $\norm{\Ab + \Ab^\top}^2$, $\tr(\Ab)$ and $\norm{\Ab_{p,:}^\top + \Ab_{:,p}}^2$.
Especially, the $\norm{\Ab + \Ab^\top}^2$ term dominates the sample complexity.
In contrast, to achieve the same precision described in Eq.~\eqref{eq:eps}, \citet{baston2022stochastic} shows that stochastic diagonal entries estimation based on matrix-vector product only needs
\begin{equation}\label{eq:N_p}
N' = \frac{2(\norm{\Ab_p}^2 - \Ab_{p,p}^2)\log\frac{2}{\delta}}{\varepsilon^2}
\end{equation}
matrix-vector product oracles.
Comparing Eq.~\eqref{eq:N} and Eq.~\eqref{eq:N_p}, we can observe that our stochastic diagonal entries estimation requires a larger query number than the one based on matrix-vector product. 
This is because Algorithm~\ref{alg:SA} can only query to the quadratic forms $\ub^\top \Ab \ub$ which contains much less information than the one of matrix product $\Ab \ub$. 

This information difference between  the quadratic forms $\ub^\top \Ab \ub$ and matrix-vector $\Ab\ub$ can be demonstrated in the following example.
Letting $\Ab$ be a Hessian matrix of $f(\xb)$, then Proposition~\ref{prop:uAu} shows that $\ub^\top \Ab \ub$ can be approximated by only three queries of the function value. 
Instead, the Hessian-vector product requires to compute the gradient $\nabla f(\xb)$ shown as follows:
\begin{equation}
	\nabla^2 f(\xb) \ub \approx \frac{\nabla f(\xb + \alpha \ub) - \nabla f(\xb - \alpha \ub) }{2\alpha}.
\end{equation}
Noting that, it may require $\cO(d)$ function queries to compute the gradient $\nabla f(\cdot)$. 
Thus, our stochastic diagonal entries estimation requires a larger query number than the one based on matrix-vector product.

Next, we will give a sample complexity of Algorithm~\ref{alg:SA} to achieve a norm-wise error bound.
\begin{theorem}\label{thm:main1}
Let $\bm{g}$ defined in Eq.~\eqref{eq:g} be the output of Algorithm~\ref{alg:SA}. 
Given parameters $0<\delta<1$ and $0<\varepsilon$, then with a probability at least $1-\delta$, it holds that
\begin{equation}\label{eq:eps3}
\norm{\bm{g} - \diag(\Ab)}^2 \leq \varepsilon\sum_{i=1}^{d} \Ab_{i,i}^2,
\end{equation}
if $N$ satisfies that
\begin{equation}\label{eq:N1}
N = \frac{1}{4\varepsilon\delta} \cdot \frac{(4d + 16) \Big(\tr(\Ab)\Big)^2
+ (d+8)\norm{\Ab + \Ab^\top}^2 
+ 20\sum_{i=1}^{d} \Ab_{i,i}^2}{\sum_{i=1}^{d} \Ab_{i,i}^2}. 
\end{equation}
\end{theorem}

Note that $\norm{\Ab + \Ab^\top}^2 = \norm{\Ab + \Ab^\top}^2 - 4 \sum_{i=1}^{d}\Ab_{i,i}^2  + 4 \sum_{i=1}^{d}\Ab_{i,i}^2 $ and $4 \sum_{i=1}^{d}\Ab_{i,i}^2 = \norm{\diag(\Ab + \Ab^\top)}^2$. 
Thus, the term $ \norm{\Ab + \Ab^\top}^2 - 4 \sum_{i=1}^{d}\Ab_{i,i}^2 $ is the square of ``Frobenius'' norm of off-diagonal entries of $\Ab + \Ab^\top$.
Accordingly, we can represent Eq.~\eqref{eq:N1} as 
\begin{align*}
	N = \frac{1}{4\varepsilon\delta} \cdot \left( \frac{(4d + 16) \Big(\tr(\Ab)\Big)^2}{\sum_{i=1}^{d} \Ab_{i,i}^2} + \frac{(d+8) \left(\norm{\Ab + \Ab^\top}^2 - 4 \sum_{i=1}^{d}\Ab_{i,i}^2\right) }{\sum_{i=1}^{d} \Ab_{i,i}^2} + \Big(4d + 52\Big) \right)
\end{align*} 
Thus, if $|\tr(\Ab)|$ is small and the off-diagonal entries of $\Ab + \Ab^\top$ are of small absolute value, then sample size $N$ required to achieve the same accuracy can also be small.
In this case, then the sample complexity is dominated by the dimension $d$.

To achieve the accuracy in Eq.~\eqref{eq:eps3}, Corollary 2 of \citet{baston2022stochastic} shows that stochastic diagonal estimation with matrix-vector product needs a sample complexity
\begin{equation}
N' = \frac{\norm{\Ab}^2 - \sum_{i=1}^{d} \Ab_{i,i}^2}{\varepsilon \cdot \sum_{i=1}^{d} \Ab_{i,i}^2} \cdot \log\frac{1}{\delta}.
\end{equation}
Comparing above equation with Eq.~\eqref{eq:N1}, we can obtain that $N = \cO(d) \cdot N'$. 
This is also because the matrix-vector product can provide more information than the matrix quadratic form.
\subsection{High Probability Version}

The sample size $N$ in Eq.~\eqref{eq:N} depends the $\frac{1}{\delta}$. 
Thus, it requires a large sample size to guarantee Eq.~\eqref{eq:eps} to hold with a high probability.
Next, we will make a simple variant of Algorithm~\ref{alg:SA} to make the sample size $N$ depend on $\log\frac{1}{\delta}$ instead of $\delta$.

The new algorithm relies on Algorithm~\ref{alg:SA} which runs   Algorithm~\ref{alg:SA}   $8\log\frac{1}{\delta}$ times   with parameter $N = \frac{V}{\varepsilon^2}$ and coordinately take the median of  these $8\log\frac{1}{\delta}$ outputs of Algorithm~\ref{alg:SA} as the final result.  
The detailed algorithm description is listed in Algorithm~\ref{alg:SA1}.
The idea behind of Algorithm~\ref{alg:SA1} is based on the following fact.
Theorem~\ref{thm:main} shows that if we take sample size $N = \frac{V}{\varepsilon^2}$, then each output of Algorithm~\ref{alg:SA} satisfies Eq.~\eqref{eq:eps} with a probability $\frac{1}{4}$. 
Then, the median of $8\log\frac{1}{\delta}$ outputs of  Algorithm~\ref{alg:SA} will satisfies Eq.~\eqref{eq:eps} with a probability $1-\delta$.

\begin{theorem}\label{thm:main2}
Let $\bm{g}$ be the output of Algorithm~\ref{alg:SA1}. 
Given parameters $0<\delta<1$ and $0<\varepsilon$, then with a probability at least $1-\delta$, 
it holds that
\begin{equation}\label{eq:eps2}
\Ab_{p,p} - \varepsilon \leq \bm{g}_p \leq \Ab_{p,p} + \varepsilon,
\end{equation}
if parameters $N'$ and $T$ of Algorithm~\ref{alg:SA1} satisfies
\begin{equation*}
N' = \frac{2 \Big(\tr(\Ab) + 4 \Ab_{p,p}\Big)^2 
+ \norm{\Ab + \Ab^\top}^2 
+ 8\norm{\Ab_{p,:} + \Ab_{:,p}}^2
- 12 \Ab_{p,p}^2}{ \varepsilon^2},
\quad\mbox{ and }\quad
T = 8\log\frac{1}{\delta}.
\end{equation*}
\end{theorem}

Theorem~\ref{thm:main2} shows that the total sample complexity to achieve Eq.~\eqref{eq:eps2} is 
\begin{align*}
N = N' T 
= 
8\log\frac{1}{\delta}  \cdot \frac{2 \Big(\tr(\Ab) + 4 \Ab_{p,p}\Big)^2 
	+ \norm{\Ab + \Ab^\top}^2 
	+ 8\norm{\Ab_{p,:} + \Ab_{:,p}}^2
	- 12 \Ab_{p,p}^2}{ \varepsilon^2}.
\end{align*}
The sample complexity in above equation depends on $\log\frac{1}{\delta}$ instead of $\frac{1}{\delta}$ in Eq.~\eqref{eq:N}. 

\begin{algorithm}[t]
\caption{Stochastic Diagonal Entries Estimation with High Probability}
\label{alg:SA1}
\begin{small}
\begin{algorithmic}[1]
\STATE {\bf Input:} A query $Q_{\Ab}(\cdot)$ which returns $Q_{\Ab}(\ub) = \ub^\top \Ab \ub$, sample size $N'$ and iteration number $T$
\FOR{$t = 1,\dots, T$}
\STATE Obtain $\bm{g}^{(t)}$ by Algorithm~\ref{alg:SA} with $Q_{\Ab}(\cdot)$ and $N'$ as input.
\ENDFOR
\STATE {\bf Output:} the median of $\bm{g}^{(t)}  $ coordinately.
\end{algorithmic}
\end{small}
\end{algorithm}

\section{Sample Complexity Analysis}

In the previous section, we provide the sample complexities of Algorithm~\ref{alg:SA} and Algorithm~\ref{alg:SA1}. 
In this section, we will provide the detailed analysis to obtain these sample complexities.
The key to the analysis is the variance of the unbiased estimation in Eq.~\eqref{eq:E}.

\subsection{Variance of Estimation}

First, the variance of our unbiased estimation has the following decomposition.
\begin{lemma}
Given a matrix $\Ab\in\RR^{d\times d}$ and a random Gaussian vector $\ub\sim \cN(\bm{0}, \I_d)$, and an integer $1\leq p\leq d$, the it holds that
\begin{equation}\label{eq:var}
\begin{aligned}
&\EE\left[ \left( \ub^\top \Ab \ub \cdot \ub_p^2 - \ub^\top \Ab \ub - 2\Ab_{p,p} \right)^2 \right]\\
=&	
\EE\left[ \left( \ub^\top \Ab \ub \cdot \ub_p^2  \right)^2 \right] 
+ \EE\left[ \left( \ub^\top \Ab \ub \right)^2 \right]
- 2 \EE\left[ \left(\ub^\top \Ab \ub \cdot \ub_p \right)^2  \right]
- 4\Ab_{p,p}^2.
\end{aligned}
\end{equation}
\end{lemma}
\begin{proof}
By Eq.~\eqref{eq:E} and Lemma~\ref{lem:var}, we can obtain that
\begin{align*}
\EE\left[ \left( \ub^\top \Ab \ub \cdot \ub_p^2 - \ub^\top \Ab \ub - 2\Ab_{p,p} \right)^2 \right]
=
\EE\left[ \left(\ub^\top \Ab \ub \cdot \ub_p^2 - \ub^\top \Ab \ub\right)^2 \right] 
- 4\Ab_{p,p}^2.
\end{align*}
Furthermore, 
\begin{align*}
\EE\left[ \left(\ub^\top \Ab \ub \cdot \ub_p^2 - \ub^\top \Ab \ub\right)^2 \right]
=
\EE\left[ \left( \ub^\top \Ab \ub \cdot \ub_p^2  \right)^2 \right] 
+ \EE\left[ \left( \ub^\top \Ab \ub \right)^2 \right]
- 2 \EE\left[ \left(\ub^\top \Ab \ub \cdot \ub_p \right)^2  \right].
\end{align*}
Combining above equation, we can obtain the result.
\end{proof}

Eq.~\eqref{eq:var} shows that the variance mainly depends on the terms $\EE\left[\left(\ub^\top \Ab \ub\cdot \bm{u}_p^n  \right)^2  \right]$ with $n = 0,1,2$.
Next, we will compute the exact value of $\EE\left[\left(\ub^\top \Ab \ub\cdot \bm{u}_p^n  \right)^2  \right]$.

\begin{lemma}\label{lem:3a}
Given a matrix $\Ab\in\RR^{d\times d}$ and a random Gaussian vector $\ub\sim \cN(\bm{0}, \I_d)$, and an integer $1\leq p\leq d$, then it holds that
\begin{equation}\label{eq:a}
\EE\left[ \left(\ub^\top \Ab \ub   \right)^2 \right]
=
\Big(\tr(\Ab)\Big)^2 + 2 \sum_{i=1}^{d} \Ab_{i,i}^2 
+ \sum_{i>j}^{d} \Big(\Ab_{i,j} + \Ab_{j,i}\Big)^2,
\end{equation}
\begin{equation}\label{eq:a1}
\begin{aligned}
&\EE\left[\left( \ub^\top \Ab \ub \cdot \ub_p \right)^2\right]
= 
\Big(\tr(\Ab)\Big)^2 
+ 4 \Ab_{p,p} \tr(\Ab) 
+ 2 \sum_{i=1}^{d} \Ab_{i,i}^2 
+ 8 \Ab_{p,p}^2 \\
&
+ \sum_{i>j}^{d} \Big(\Ab_{i,j} + \Ab_{j,i}\Big)^2
+ 2 \sum_{i>j,  j= p}^{d}\Big(\Ab_{i,p} + \Ab_{p,i} \Big)^2
+ 2 \sum_{i>j,  i= p}^{d}\Big( \Ab_{p,j} + \Ab_{j,p} \Big)^2,
\end{aligned}
\end{equation}
and
\begin{equation}\label{eq:a2}
\begin{aligned}
&\EE\left[ \left(\ub^\top \Ab \ub\cdot \bm{u}_p^2  \right)^2 \right] 
=
3 \Big(\tr(\Ab)\Big)^2 + 6 \sum_{i= 1}^{d} \Ab_{i,i}^2 + 24 \Ab_{p,p} \cdot \tr(\Ab) + 72 \Ab_{p,p}^2  \\
&+ 3 \sum_{i>j}^{d} \Big(\Ab_{i,j} + \Ab_{j,i}\Big)^2
+ 12 \sum_{i>j,  j= p}^{d} \Big(\Ab_{i,p} + \Ab_{p,i}\Big)^2
+ 12 \sum_{i>j, i=p}^{d} \Big(\Ab_{p,j} + \Ab_{j, p}\Big)^2.
\end{aligned}
\end{equation}
\end{lemma}

Given above two lemmas, we can obtain the explicit formula of the variance.

\begin{lemma}
Given a matrix $\Ab\in\RR^{d\times d}$ and a random Gaussian vector $\ub\sim \cN(\bm{0}, \I_d)$, and an integer $1\leq p\leq d$, then it holds that
\begin{equation}\label{eq:main}
\begin{aligned}
&\EE\left[ \left( \ub^\top \Ab \ub \cdot \ub_p^2 - \ub^\top \Ab \ub - 2\Ab_{p,p} \right)^2 \right]\\
=&
2 \Big(\tr(\Ab) + 4 \Ab_{p,p}\Big)^2 
+ \norm{\Ab + \Ab^\top}^2 
+ 8\norm{\Ab_{p,:} + \Ab_{:,p}}^2
- 12 \Ab_{p,p}^2.
\end{aligned}	
\end{equation}
and
\begin{equation}\label{eq:main1}
\begin{aligned}
&\EE\left[ \sum_{p=1}^{d}\left(\ub^\top \Ab \ub \cdot \ub_p^2 - \ub^\top \Ab \ub - 2\Ab_{p,p}\right)^2 \right]\\
=&
(4d + 16) \Big(\tr(\Ab)\Big)^2
+ d\norm{\Ab + \Ab^\top}^2 
+ 8\sum_{p=1}^{d}\norm{\Ab_{p,:} + \Ab_{:,p}}^2
+ 20\sum_{i=1}^{d} \Ab_{i,i}^2.
\end{aligned}
\end{equation}
\end{lemma}
\begin{proof}
First, we have
\begin{align*}
&\EE\left[ \left(\ub^\top \Ab \ub \cdot \ub_p^2 - \ub^\top \Ab \ub - 2\Ab_{p,p}\right)^2 \right]\\
\stackrel{\eqref{eq:var}}{=}& 
\EE\left[ \left( \ub^\top \Ab \ub \cdot \ub_p^2  \right)^2 \right] 
+ \EE\left[ \left( \ub^\top \Ab \ub \right)^2 \right]
- 2 \EE\left[ \left(\ub^\top \Ab \ub \right)^2 \cdot \ub_p^2 \right] - 4\Ab_{p,p}^2\\
\stackrel{\eqref{eq:a}\eqref{eq:a1}\eqref{eq:a2}}{=}& 
3 \Big(\tr(\Ab)\Big)^2 + 6 \sum_{i= 1}^{d} \Ab_{i,i}^2 + 24 \Ab_{p,p} \cdot \tr(\Ab) + 68 \Ab_{p,p}^2  \\
&+ 3 \sum_{i>j}^{d} \Big(\Ab_{i,j} + \Ab_{j,i}\Big)^2
+ 12 \sum_{i>j,  j= p}^{d} \Big(\Ab_{i,p} + \Ab_{p,i}\Big)^2
+ 12 \sum_{i>j, i=p}^{d} \Big(\Ab_{p,j} + \Ab_{j, p}\Big)^2\\
&
+ \Big(\tr(\Ab)\Big)^2 + 2 \sum_{i=1}^{d} \Ab_{i,i}^2 
+ \sum_{i>j}^{d} \Big(\Ab_{i,j} + \Ab_{j,i}\Big)^2\\
&
- 2 \left(\Big(\tr(\Ab)\Big)^2 
+ 4 \Ab_{p,p} \tr(\Ab) 
+ 2 \sum_{i=1}^{d} \Ab_{i,i}^2 
+ 8 \Ab_{p,p}^2\right)\\
&
-2\left(\sum_{i>j}^{d} \Big(\Ab_{i,j} + \Ab_{j,i}\Big)^2
+ 2 \sum_{i>j,  j= p}^{d}\Big(\Ab_{i,p} + \Ab_{p,i} \Big)^2
+ 2 \sum_{i>j,  i= p}^{d}\Big( \Ab_{p,j} + \Ab_{j,p} \Big)^2\right)\\
=&
2\Big(\tr(\Ab)\Big)^2
+ 4\sum_{i= 1}^{d} \Ab_{i,i}^2
+ 16\Ab_{p,p} \cdot \tr(\Ab)
+ 52 \Ab_{p,p}^2\\
&
+ 2\sum_{i>j}^{d} \Big(\Ab_{i,j} + \Ab_{j,i}\Big)^2
+ 8 \sum_{i>j,  j= p}^{d} \Big(\Ab_{i,p} + \Ab_{p,i}\Big)^2
+ 8\sum_{i>j, i=p}^{d} \Big(\Ab_{p,j} + \Ab_{j, p}\Big)^2.
\end{align*}

Furthermore, it holds that
\begin{align*}
4\sum_{i= 1}^{d} \Ab_{i,i}^2 + 2\sum_{i>j}^{d} \Big(\Ab_{i,j} + \Ab_{j,i}\Big)^2 = \norm{\Ab + \Ab^\top}^2.
\end{align*}
We also have
\begin{align*}
8 \sum_{i>j,  j= p}^{d} \Big(\Ab_{i,p} + \Ab_{p,i}\Big)^2
+ 8\sum_{i>j, i=p}^{d} \Big(\Ab_{p,j} + \Ab_{j, p}\Big)^2 
+ 32 \Ab_{p,p}^2 
= 8\norm{\Ab_{p,:} + \Ab_{:,p}}^2
\end{align*}

Thus,
\begin{align*}
&\EE\left[ \left(\ub^\top \Ab \ub \cdot \ub_p^2 - \ub^\top \Ab \ub - 2\Ab_{p,p} \right)^2 \right]\\
=&
2\Big(\tr(\Ab)\Big)^2 
+ 16\Ab_{p,p} \cdot \tr(\Ab)
+ \norm{\Ab + \Ab^\top}^2 
+ 8\norm{\Ab_{p,:} + \Ab_{:,p}}^2
+ 20\Ab_{p,p}^2\\
=&
2 \Big(\tr(\Ab) + 4 \Ab_{p,p}\Big)^2 
+ \norm{\Ab + \Ab^\top}^2 
+ 8\norm{\Ab_{p,:} + \Ab_{:,p}}^2
- 12 \Ab_{p,p}^2.
\end{align*}

Furthermore, we can obtain that
\begin{align*}
&\EE\left[ \sum_{p=1}^{d}\left(\ub^\top \Ab \ub \cdot \ub_p^2 - \ub^\top \Ab \ub - 2\Ab_{p,p}\right)^2 \right]\\
=&
2 \sum_{p=1}^{d}\Big(\tr(\Ab) + 4 \Ab_{p,p}\Big)^2 
+ d\norm{\Ab + \Ab^\top}^2 
+ 8\sum_{p=1}^{d}\norm{\Ab_{p,:} + \Ab_{:,p}}^2
- 12 \sum_{p=1}^{d}\Ab_{p,p}^2\\
=& 
4d \Big(\tr(\Ab)\Big)^2 
+ 16 \tr(\Ab) \sum_{p=1}^{d} \Ab_{p,p} 
+ 32 \sum_{p=1}^{d}\Ab_{p,p}^2
+ d\norm{\Ab + \Ab^\top}^2 
+ 8\sum_{p=1}^{d}\norm{\Ab_{p,:} + \Ab_{:,p}}^2
- 12 \sum_{p=1}^{d}\Ab_{p,p}^2\\
=& 
(4d + 16) \Big(\tr(\Ab)\Big)^2
+ d\norm{\Ab + \Ab^\top}^2 
+ 8\sum_{p=1}^{d}\norm{\Ab_{p,:} + \Ab_{:,p}}^2
+ 20\sum_{i=1}^{d} \Ab_{i,i}^2.
\end{align*}


\end{proof}

\subsection{Sample Complexity Analysis}

Given the variance of unbiased estimation in Eq.~\eqref{eq:E}, combining with the Chebyshev's inequality, we can prove the results in Theorem~\ref{thm:main} as follows.
\begin{proof}[Proof of Theorem~\ref{thm:main}]
By Chebyshev's inequality, we can obtain that
\begin{align*}
&\PP\left\{ \left| \frac{1}{2N}\sum_{j=1}^{N} \left[\ub^{(j)}\right]^\top \Ab \ub^{(j)}\cdot \Big([\ub_p^{(j)}]^2 - 1\Big) - \Ab_{p,p} \right| > \varepsilon \right\}\\
\leq& 
\frac{1}{4N} \cdot  \frac{\EE\left[\left(\ub^\top \Ab \ub\cdot \bm{u}_p^2 -\tr(\Ab) - 2\Ab_{p,p} \right)^2\right]}{\varepsilon^2}\\
\stackrel{\eqref{eq:main}}{=}&
\frac{2 \Big(\tr(\Ab) + 4 \Ab_{p,p}\Big)^2 
+ \norm{\Ab + \Ab^\top}^2 
+ 8\norm{\Ab_{p,:} + \Ab_{:,p}}^2
- 12 \Ab_{p,p}^2}{4N \varepsilon^2},
\end{align*}
where the first inequality is because of the  Chebyshev's inequality and the independence of $\ub^{(j)}$'s in Algorithm~\ref{alg:SA}.

By setting the right hand of above equation to $\delta$, we should choose $N$ to be 
\begin{equation*}
N = \frac{2 \Big(\tr(\Ab) + 4 \Ab_{p,p}\Big)^2 
+ \norm{\Ab + \Ab^\top}^2 
+ 8\norm{\Ab_{p,:} + \Ab_{:,p}}^2
- 12 \Ab_{p,p}^2}{4\delta \varepsilon^2}.
\end{equation*}
\end{proof}
Next, we will prove the results in Theorem~\ref{thm:main1} similar to the one of Theorem~\ref{thm:main}.
\begin{proof}[Proof of Theorem~\ref{thm:main1}]
By the Markov's inequality, we can obtain that
\begin{align*}
&\PP\left\{ 	\norm{\bm{g} - \diag(\Ab)}^2 > \varepsilon\sum_{i=1}^{d} \Ab_{i,i}^2 \right\}
\leq 
\frac{\EE\left[\norm{\bm{g} - \diag(\Ab)}^2  \right]}{\varepsilon\sum_{i=1}^{d} \Ab_{i,i}^2}\\
=&
\frac{1}{4N} \frac{\EE\left[ \sum_{p=1}^{d}\left( \ub^\top \Ab \ub \cdot \ub_p^2 - \ub^\top \Ab \ub - 2\Ab_{p,p} \right)^2 \right]}{\varepsilon\sum_{i=1}^{d} \Ab_{i,i}^2}\\
\stackrel{\eqref{eq:main1}}{=}&
\frac{1}{4N} \cdot \frac{(4d + 16) \Big(\tr(\Ab)\Big)^2
+ d\norm{\Ab + \Ab^\top}^2 
+ 8\sum_{p=1}^{d}\norm{\Ab_{p,:} + \Ab_{:,p}}^2
+ 20\sum_{i=1}^{d} \Ab_{i,i}^2}{\varepsilon\sum_{i=1}^{d} \Ab_{i,i}^2}.
\end{align*}
Furthermore, it holds that
\begin{align*}
\sum_{p=1}^{d}\norm{\Ab_{p,:} + \Ab_{:,p}}^2 
= 
\norm{ \Ab + \Ab^\top }^2.
\end{align*}
Thus,
\begin{align*}
\PP\left\{ 	\norm{\bm{g} - \diag(\Ab)}^2 > \varepsilon\sum_{i=1}^{d} \Ab_{i,i}^2 \right\}
\leq
\frac{1}{4N} \cdot \frac{(4d + 16) \Big(\tr(\Ab)\Big)^2
	+ (d+8)\norm{\Ab + \Ab^\top}^2 
	+ 20\sum_{i=1}^{d} \Ab_{i,i}^2}{\varepsilon\sum_{i=1}^{d} \Ab_{i,i}^2}.
\end{align*}

By setting the right hand of above equation to $\delta$, we should choose $N$ to be 
\begin{align*}
N = \frac{1}{4\varepsilon\delta} \cdot \frac{(4d + 16) \Big(\tr(\Ab)\Big)^2
+ (d+8)\norm{\Ab + \Ab^\top}^2 
+ 20\sum_{i=1}^{d} \Ab_{i,i}^2}{\sum_{i=1}^{d} \Ab_{i,i}^2}. 
\end{align*}	
\end{proof}

Next, we will prove the results in Theorem~\ref{thm:main2}.

\begin{proof}[Proof of Theorem~\ref{thm:main2}]
By the setting of $N'$ and Theorem~\ref{thm:main}, we can obtain that for any integer $1\leq p\leq d$, with a probability at least $1-\frac{1}{4}$, $\bm{g}_p^{(j)}$ in Algorithm~\ref{alg:SA1} satisfies that
\begin{equation}\label{eq:eps1}
\Ab_{p,p} - \varepsilon \leq \bm{g}_p^{(j)} \leq \Ab_{p,p} +\varepsilon.
\end{equation}

Define a random variable $Y^{(j)}$
\begin{equation}
Y^{(j)} 
=
\begin{cases}
1, \qquad\mbox{if $\bm{g}_p^{(j)}$ satisfies Eq.~\eqref{eq:eps1}}\\
0, \qquad\mbox{otherwise}
\end{cases}.
\end{equation} 

Assume that $\bm{g}_p$, the output of Algorithm~\ref{alg:SA1}, is a bad estimation, that is $\bm{g}_p > \Ab_{p,p} +\varepsilon$ or $\bm{g}_p < \Ab_{p,p} - \varepsilon$. 
Because $\bm{g}_p$ is the median of $\bm{g}_p^{(j)}$, if $\bm{g}_p > \Ab_{p,p} +\varepsilon$, then at most half of $\bm{g}_p^{(j)}$ will no larger than $\Ab_{p,p} +\varepsilon$ which implies that at most half of $\bm{g}_p^{(j)}$ satisfy Eq.~\eqref{eq:eps1}.
For the same reason, if $\bm{g}_p < \Ab_{p,p} -\varepsilon$, then at most half of $\bm{g}_p^{(j)}$ will no less than $\Ab_{p,p} -\varepsilon$ which implies that at most half of $\bm{g}_p^{(j)}$ satisfy Eq.~\eqref{eq:eps1}.
Therefore,  $|\bm{g}_p - \Ab_{p,p}|>\varepsilon$ implies that    at most half of $\bm{g}_p^{(j)}$ satisfy Eq.~\eqref{eq:eps1} which will lead to $\sum_{j} Y^{(j)} \leq \frac{T}{2}$.
Accordingly, we can obtain that
\begin{align*}
\PP\left( |\bm{g}_p - \Ab_{p,p}|>\varepsilon \right) 
\leq& 
\PP\left( \sum_{j} Y^{(j)} \leq \frac{T}{2}  \right)
=
\PP\left( \sum_{j} Y^{(j)} - \EE\left[ \sum_{j} Y^{(j)} \right] \leq \frac{T}{2} - \EE\left[ \sum_{j} Y^{(j)} \right] \right)\\
\leq& 
\PP\left( \sum_{j} Y^{(j)} - \EE\left[ \sum_{j} Y^{(j)} \right] \leq - \frac{T}{4} \right)
\leq
\exp\left( \frac{-2 T^2/16}{T} \right),
\end{align*}
where the second inequality is because $\EE\left[\sum_{j}Y^{(j)}\right] = \sum_{j}\EE\left[ Y^{(j)} \right] \geq \frac{3T}{4}$ and the last inequality is because of Lemma~\ref{lem:hoeffding}.

By setting $\exp\left( \frac{-2 T^2/16}{T} \right) = \delta$, we can obtain that $T = 8\log\frac{1}{\delta}$ which concludes the proof.
\end{proof}

\begin{figure*}[!ht]
	\begin{center}
		\centering
		\subfigure[\textsf{ Element-wise relative error with $p = 1$.}]{\includegraphics[width=55mm]{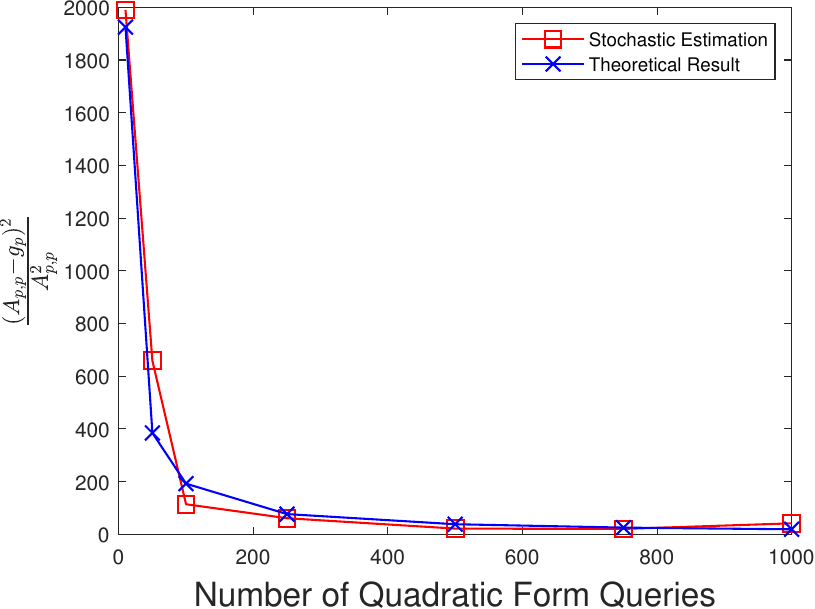}}~
		\subfigure[\textsf{ Element-wise relative error with $p = \argmax_p |\Ab_{p,p}|$.}]{\includegraphics[width=55mm]{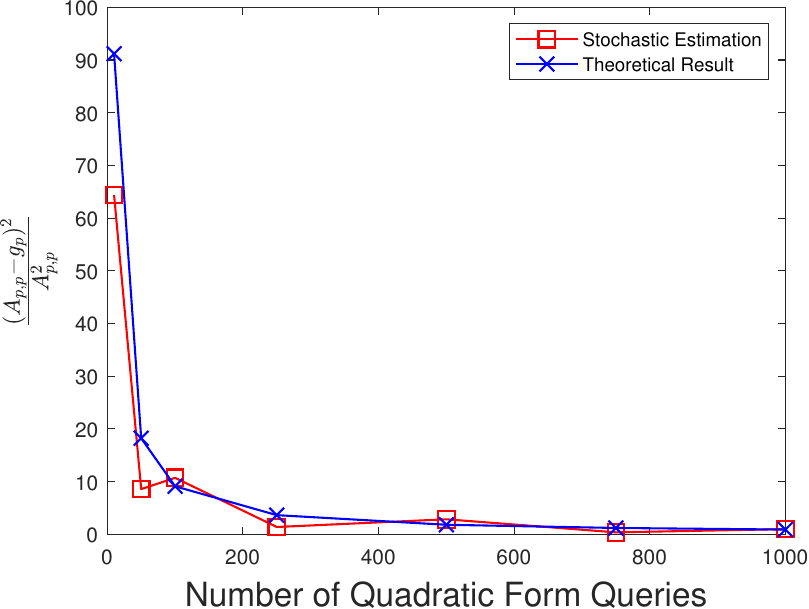}}\\
		\subfigure[\textsf{ Element-wise relative error with $p = \argmin_p |\Ab_{p,p}|$.}]{\includegraphics[width=55mm]{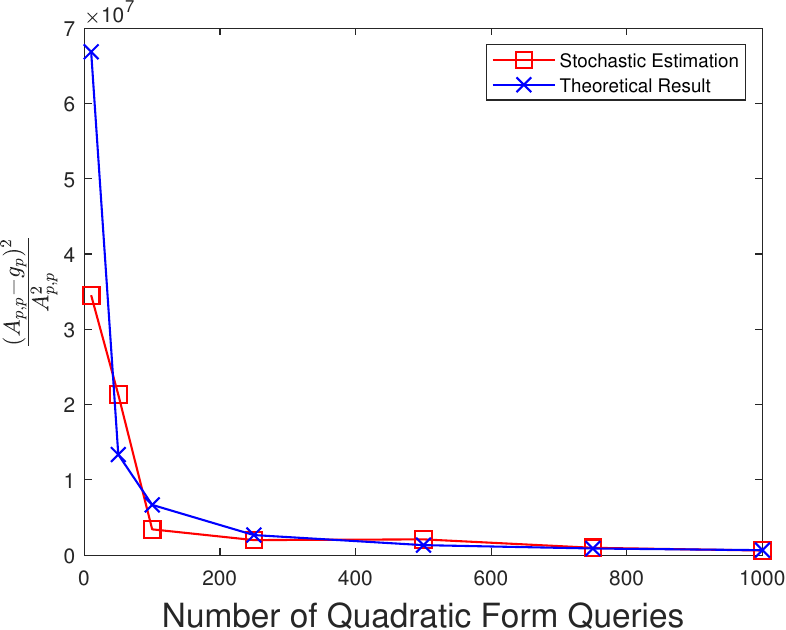}}~	
		\subfigure[\textsf{ Norm-wise relative error.}]{\includegraphics[width=55mm]{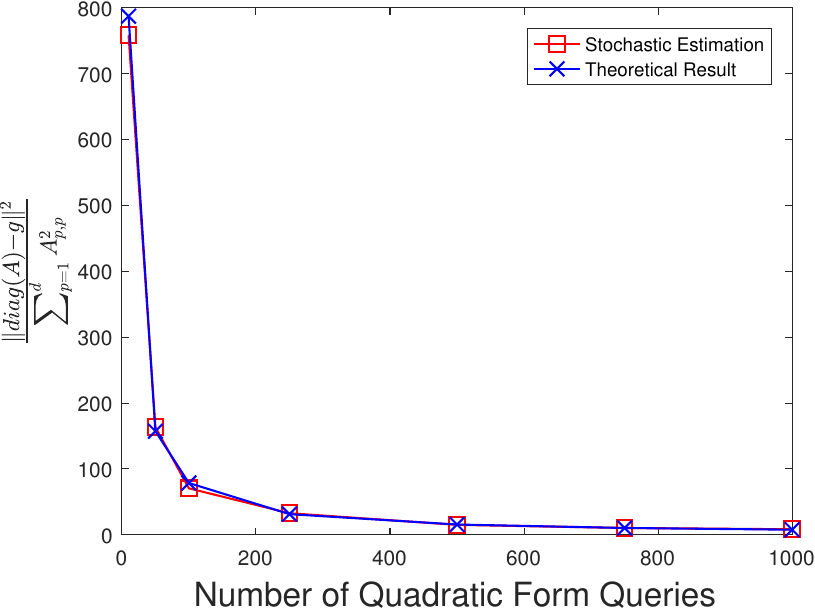}}	
	\end{center}
	\caption{Experiment results on random Gaussian matrix.}
	\label{fig:result_gauss}
\end{figure*}

\section{Experiments}
\label{sec:exp}

In the previous section, we propose Algorithm~\ref{alg:SA} to estimate the diagonal entries of a matrix by querying its quadratic form and provide the sample complexity of Algorithm~\ref{alg:SA}.
In this section, we will empirically study the performance of Algorithm~\ref{alg:SA} and validate our sample complexity analysis.

\subsection{Test Matrices}

We will perform numerical experiments on three  different types of matrices to illustrate the accuracy of Algorithm~\ref{alg:SA}.
The first type matrix is a $100\times 100$ random Gaussian matrix. 
Because  each entry of the matrix follows the standard Gaussian distribution, then its entries can be larger or less than zero and it is an asymmetric matrix.
The second type matrix is a $100\times 100$ random matrix with each entry uniformly drawing from $[0, 1]$. 
This matrix is also asymmetric but with all positive entries.
The third matrix is a real-word positive definite matrix ``Boeing msc10480'' of dimension $10480\times 10480$ \citep{davis2011university}.    

\begin{figure*}[!ht]
	\begin{center}
		\centering
		\subfigure[\textsf{Element-wise error relative with $p = 1$.}]{\includegraphics[width=55mm]{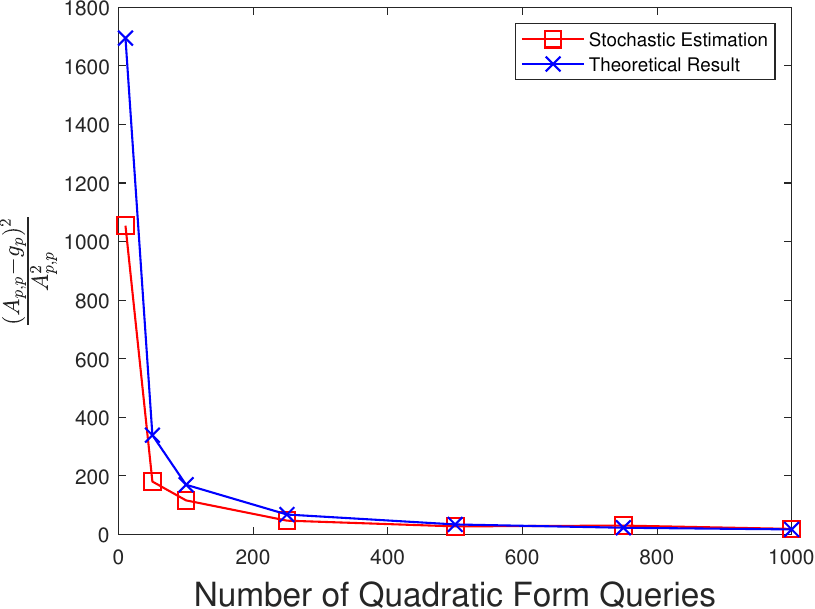}}~
		\subfigure[\textsf{Element-wise relative error with $p = \argmax_p |\Ab_{p,p}|$.}]{\includegraphics[width=55mm]{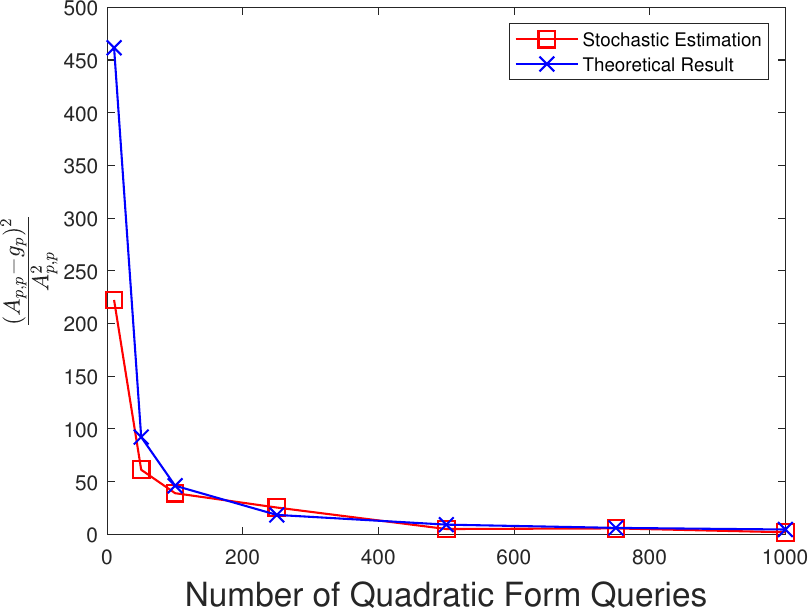}}\\
		\subfigure[\textsf{Element-wise relative error with $p = \argmin_p |\Ab_{p,p}|$.}]{\includegraphics[width=55mm]{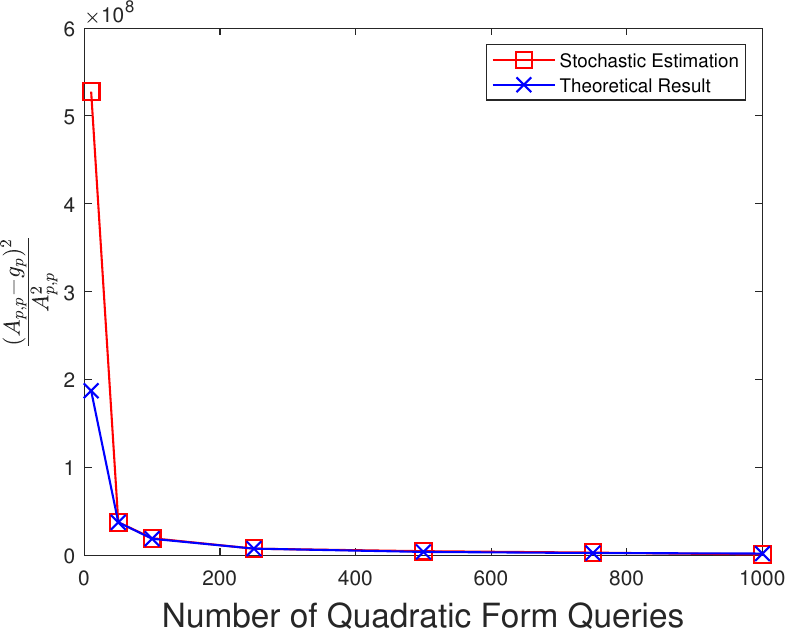}}~	
		\subfigure[\textsf{ Norm-wise relative error.}]{\includegraphics[width=55mm]{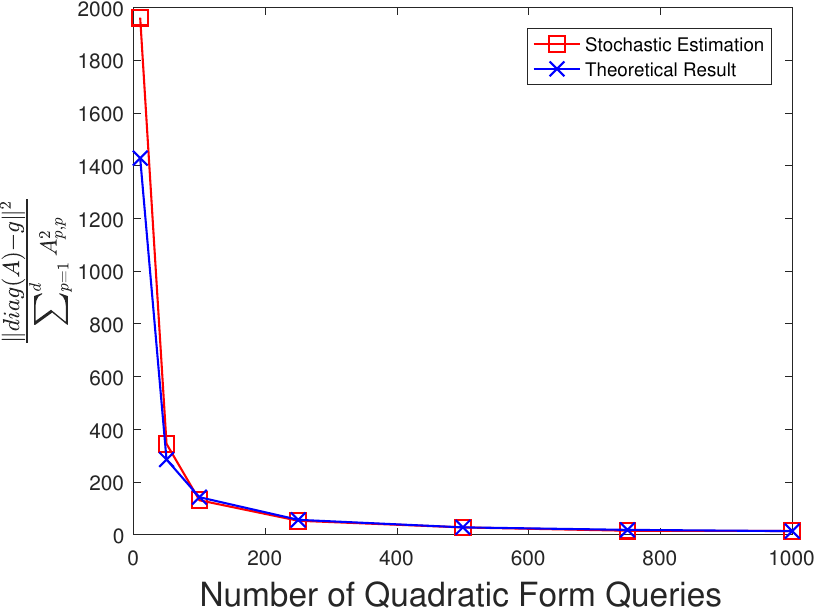}}	
	\end{center}
	\caption{Experiment results on random matrix with each entry uniformly drawing from $[0, 1]$.}
	\label{fig:result_rand}
\end{figure*}

\subsection{Experiment Settings}

In our experiment, we will measure the element-wise relative error defined as follows: 
\begin{equation}
	\varepsilon_p = \frac{(\Ab_{p,p} - \bm{g}_p)^2}{\Ab_{p,p}^2}.
\end{equation}
Noting that achieve the above element-wise relative error, we only need to set value $\varepsilon = \varepsilon_p^{1/2} |\Ab_{p,p}|$ in Eq.~\eqref{eq:eps}. 
According to Eq.~\eqref{eq:N}, given a sample size $\widehat{N}$, the element-wise relative error should be 
\begin{equation}\label{eq:eps_1}
\hat{\varepsilon}_p = \frac{2 \Big(\tr(\Ab) + 4 \Ab_{p,p}\Big)^2 
	+ \norm{\Ab + \Ab^\top}^2 
	+ 8\norm{\Ab_{p,:}^\top + \Ab_{:,p}}^2
	- 12 \Ab_{p,p}^2}{4\delta \widehat{N} \Ab_{p,p}^2}
\end{equation} 
Because vector $\bm{g}$ has $d$ different entries,  we will report the case $p = 1$, $p = \argmax_p |\Ab_{p,p}|$ and $p = \argmin_p |\Ab_{p,p}|$.

Furthermore, we will report the norm-wise relative error 
\begin{equation}\label{eq:eps_pp}
	\varepsilon = \frac{\norm{\bm{g} - \diag(\Ab)}^2}{\sum_{p=1}^{d} \Ab_{p,p}^2}.
\end{equation}
By Theorem~\ref{thm:main1}, given a sample size $\widehat{N}$, the norm-wise relative error should be 
\begin{equation}\label{eq:eps_2}
\hat{\varepsilon} = \frac{1}{4\widehat{N}\delta} \cdot \frac{(4d + 16) \Big(\tr(\Ab)\Big)^2
	+ (d+8)\norm{\Ab + \Ab^\top}^2 
	+ 20\sum_{i=1}^{d} \Ab_{i,i}^2}{\sum_{i=1}^{d} \Ab_{i,i}^2}.
\end{equation}

For the two $100\times 100$ matrices, we set sample size $N = [10, 50, 100, 250, 500, 750, 1000]$. 
For the ``Boeing msc10480'' matrix, we set sample size $N = [100,  1000, 5000, 10000, 50000,100000]$.
For each sample size, we will run Algorithm~\ref{alg:SA} ten times and compute the mean of these ten relative errors.  
In our experiment, we will report the empirical relative errors.
Furthermore, to validate the tightness of theoretical sample complexities, we will also report  the element-wise and norm-wise relative errors computed as Eq.~\eqref{eq:eps_1} and Eq.~\eqref{eq:eps_2} with $\delta = 1$ and compare them with the empirical relative errors.  

\begin{figure*}[!ht]
	\begin{center}
		\centering
		\subfigure[\textsf{Element-wise relative error with $p = 1$.}]{\includegraphics[width=55mm]{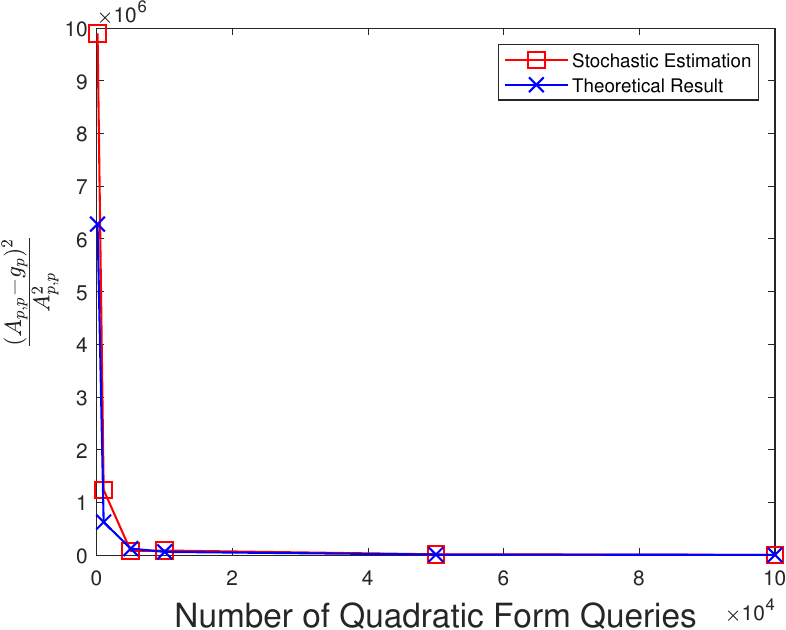}}~
		\subfigure[\textsf{Element-wise relative error with $p = \argmax_p |\Ab_{p,p}|$.}]{\includegraphics[width=55mm]{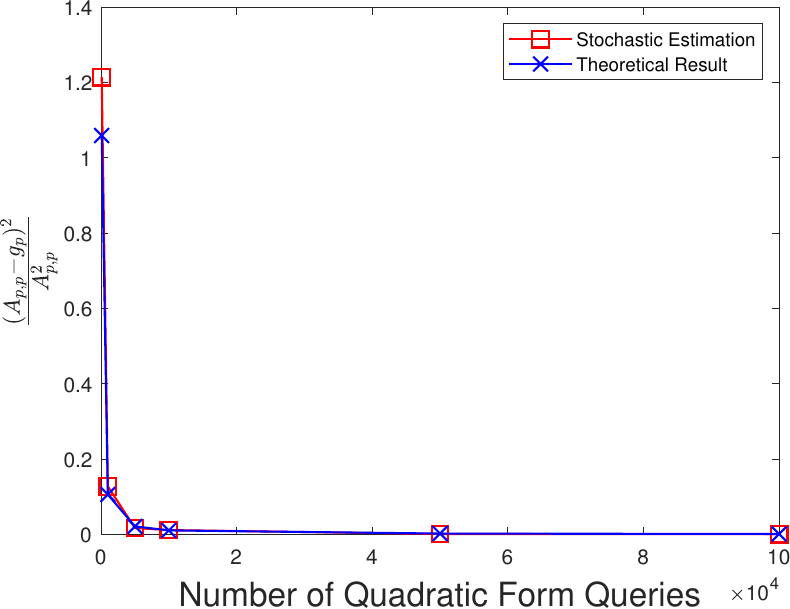}}\\
		\subfigure[\textsf{Element-wise relative error with $p = \argmin_p |\Ab_{p,p}|$.}]{\includegraphics[width=55mm]{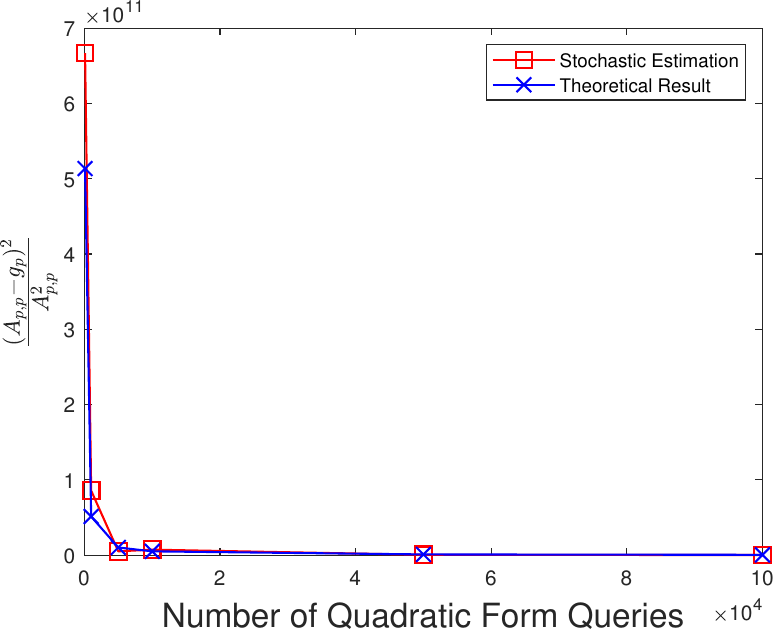}}~	
		\subfigure[\textsf{ Norm-wise relative error.}]{\includegraphics[width=55mm]{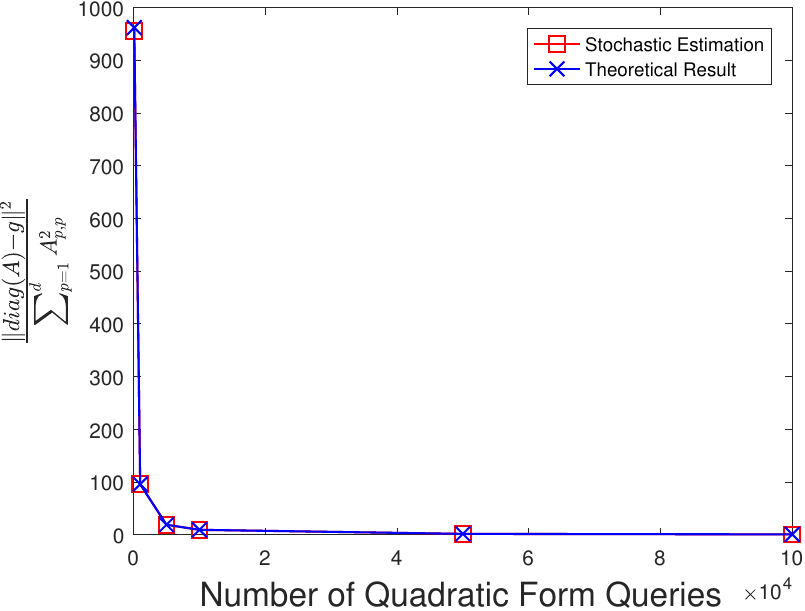}}	
	\end{center}
	\caption{Experiment results on ``Boeing msc10480'' matrix.}
	\label{fig:result_msc}
\end{figure*}

\subsection{Experiment Results}
We report our experiment results in Fig.~\ref{fig:result_gauss}-Fig.~\ref{fig:result_msc}. 
First, all experiment results show that a large sample size will lead to a small relative error. 
Furthermore, with the same sample size, Fig.~\ref{fig:result_gauss}-(b) and Fig.~\ref{fig:result_gauss}-(c) show that $|\Ab_{p,p}|$ being of  large value can achieve much smaller relative errors than the one of $|\Ab_{p,p}|$ being of small value. 
This is because when fixing a sample size $\widehat{N}$, then the element-wise relative error $\hat{\varepsilon}_p$ is of order $\cO\left(\frac{\norm{\Ab+\Ab^\top}^2}{\Ab_{p,p}^2}\right)$ just as shown in Eq.~\eqref{eq:eps_1}.

Comparing our theoretical relative errors in Eq.~\eqref{eq:eps_1} and Eq.~\eqref{eq:eps_2} with experiment results, we can observe that when the sample size is small, there are large gaps between our theoretical relative error and  experiment results. 
However, these gaps decrease as the sample size increases. 
We conjecture that this phenomenon is because when the sample size is small, the variance is too large such that repeating ten times is also not enough. 
For the large sample size, we can observe that our theoretical relative errors are very close to the experiment results.
This can effectively validate the tightness of theoretical sample complexity analysis.
Especially,  when the sample size is large, Fig.~\ref{fig:result_gauss}-(d), Fig.~\ref{fig:result_rand}-(d) and Fig.~\ref{fig:result_msc}-(d) show that our theoretical norm-wise relative errors  are almost the same to the experiment results. 

Fig.~\ref{fig:result_msc}-(b) shows  our algorithm with only $100$ queries to matrix quadratic form can achieve a $1.2$-relative error for the largest diagonal entry when the matrix is of dimension  $10480\times 10480$.  
This implies that our algorithm is an effective diagonal estimation method even for high dimensional matrices.

\section{Conclusion}

In this paper, we propose a stochastic diagonal estimation method with queries to the matrix quadratic form. 
We provide the element-wise and norm-wise sample complexities. 
Our numerical experiments demonstrate the effectiveness of our method and validate the tightness of theoretical results. 
\pb
\bibliography{ref.bib}
\bibliographystyle{apalike2}

\appendix

\section{Some Useful Lemmas}

\begin{lemma}\label{lem:var}
Letting $a$ be a random variable, then it holds that $\EE\left[\Big(a - \EE[a]\Big)^2\right] = \EE\left[a^2\right] - \Big(\EE\left[a\right]\Big)^2$.
\end{lemma}

\begin{lemma}
Letting  $u \sim \cN(0,1)$, that is $u$ is a random variable of standard Gaussian distribution, then it holds that 
\begin{equation}\label{eq:Eu}
\EE\left[u^2\right] = 1,\; \EE\left[u^4\right] = 3,\; \EE\left[u^6\right] = 15,\;\mbox{and}\;\; \EE\left[u^8\right] = 105,
\end{equation} 
and 
\begin{equation}\label{eq:zero}
\EE\left[u^n\right] = 0, \forall n\mbox{ is an odd integer.}
\end{equation}
\end{lemma}

\begin{lemma}[\cite{hoeffding1994probability}]\label{lem:hoeffding}
Let $X_1,\dots , X_n$ be independent bounded random variables such
that $X_i$ falls in the interval $[a_i, b_i]$ with probability one. Letting $S_n = \sum_{i=1}^{n} X_i$,  then for any $t > 0$
we have
\begin{equation}
\PP\left( S_n - \EE[S_n] \leq -t \right) \leq \exp\left(\frac{-2t^2}{\sum_{i=1}^{n}(b_i - a_i)^2}\right)
\end{equation}
\end{lemma}

\section{Proof of Proposition~\ref{prop:uAu}}

\begin{proof}[Proof of Proposition~\ref{prop:uAu}]
	First, by the definition of $\Delta(\xb, \alpha\ub)$, we can obtain the following equations:
	\begin{align*}
		f(\xb + \alpha \ub) =& f(\xb) + \alpha \nabla^\top f(\xb) \ub + \frac{\alpha^2}{2} \ub^\top \nabla^2 f(\xb) \ub + \Delta(\xb, \alpha\ub)\\
		f(\xb - \alpha \ub) =& f(\xb) - \alpha \nabla^\top f(\xb) \ub + \frac{\alpha^2}{2} \ub^\top \nabla^2 f(\xb) \ub + \Delta(\xb, -\alpha\ub).
	\end{align*}
Thus, we can obtain that 
	\begin{align*}
		\frac{f(\xb + \alpha \ub) + f(\xb - \alpha \ub) - 2f(\xb)}{\alpha^2}
		=
		\ub^\top \nabla^2 f(\xb) \ub + \frac{\Delta(\xb, \alpha\ub) + \Delta(\xb, -\alpha\ub)}{\alpha^2}
	\end{align*}

	By the $\gamma$-Lipschitz continuous of the Hessian, Lemma~1 of \citet{nesterov2006cubic} shows that
	\begin{align*}
		\left|f(\yb) - f(\xb) - \nabla^\top f(\xb) (\yb - \xb) - \frac{1}{2} (\yb - \xb)^\top \nabla^2 f(\xb) (\yb - \xb) \right| \leq \frac{\gamma \norm{\yb - \xb}^3}{6}.
	\end{align*}
	Then, we can obtain that
	\begin{align*}
		\Big|\Delta(\xb, \alpha\ub)\Big| \leq \frac{\alpha^3 \norm{\ub}^3}{6} \qquad\mbox{ and }\qquad \Big|\Delta(\xb, -\alpha\ub)\Big| \leq \frac{\alpha^3 \norm{\ub}^3}{6}.
	\end{align*}
	Thus, 
	\begin{align*}
		\left| \frac{\Delta(\xb, \alpha\ub) + \Delta(\xb, -\alpha\ub)}{\alpha^2} \right|
		\leq 
		\frac{\alpha \norm{\ub}^3}{3}.
	\end{align*}
	
	Furthermore, by Lemma~4 of \citet{ye2025hessian}, with a probability $1-\delta$,  it holds that $\norm{\ub} \leq 2d + 3 \log\frac{1}{\delta}$.
	Thus, we can obtain that 
	\begin{equation*}
		\left| \frac{\Delta(\xb, \alpha\ub) + \Delta(\xb, -\alpha\ub)}{\alpha^2} \right| \leq \frac{\alpha \left( 2d + 3 \log\frac{1}{\delta} \right)^3}{3}.
	\end{equation*}
\end{proof}

\section{Proof of Lemma~\ref{lem:3a}}

First, we decompose $\EE\left[ \left(\ub^\top \Ab \ub\cdot \bm{u}_p^n  \right)^2 \right]$ into  two terms.
\begin{lemma}
Given a matrix $\Ab\in\RR^{d\times d}$ and a random Gaussian vector $\ub\sim \cN(\bm{0}, \I_d)$, and an integer $1\leq p\leq d$, then it holds that
\begin{equation}\label{eq:uAu}
\begin{aligned}
\EE\left[ \left(\ub^\top \Ab \ub\cdot \bm{u}_p^n  \right)^2 \right]
=
\EE\left[ \left(  \sum_{i=1}^{d} \Ab_{i,i} \ub_i^2 \ub_p^n\right)^2 \right] + \EE\left[ \left(  \sum_{i\neq j}^{d} \Ab_{i,j} \ub_i\ub_j \ub_p^n \right)^2 \right].
\end{aligned}
\end{equation}
\end{lemma}
\begin{proof}
We have
\begin{align*}
&\EE\left[ \left(\ub^\top \Ab \ub\cdot \ub_p^n  \right)^2 \right]
=
\EE\left[ \left(\sum_{i, j=1} \Ab_{i,j} \ub_i\ub_j \ub_p^n\right)^2  \right] \\
=&
\EE\left[ \left(\sum_{i=1}^{d} \Ab_{i,i} \ub_i^2 \ub_p^n +  \sum_{i\neq j}^{d} \Ab_{i,j} \ub_i\ub_j \ub_p^n\right)^2 \right]\\
=&
\EE\left[ \left(  \sum_{i=1}^{d} \Ab_{i,i} \ub_i^2 \ub_p^n\right)^2 \right] + \EE\left[ \left(  \sum_{i\neq j}^{d} \Ab_{i,j} \ub_i\ub_j \ub_p^n \right)^2 \right]\\
&+ 2\EE\left[ \left( \sum_{i=1}^{d} \Ab_{i,i} \ub_i^2 \ub_p^n \right) \cdot \left( \sum_{i \neq  j}^{d} \Ab_{i,j} \ub_i\ub_j \ub_p^n \right) \right].
\end{align*}
Furthermore, 
\begin{align*}
\EE\left[ \left( \sum_{i=1}^{d} \Ab_{i,i} \ub_i^2 \ub_p^n \right) \cdot \left( \sum_{i \neq  j}^{d} \Ab_{i,j} \ub_i\ub_j \ub_p^n \right) \right]
=
\EE\left[ \sum_{i=1}^{d} \sum_{k\neq \ell}^{d}\Ab_{i,i} \Ab_{k,\ell} \ub_i^2 \ub_k\ub_\ell \ub_p^{2n} \right]
= 0,
\end{align*}
where the last equality is because of Eq.~\eqref{eq:zero}.
Thus, we can obtain the final result.
\end{proof}

Next, we will bound two terms in the right hand of Eq.~\eqref{eq:uAu} in the next two subsections.

\subsection{Bound of First Term}

We will bound the first term in Eq.~\eqref{eq:uAu} and we have the following result.
\begin{lemma}
Given a matrix $\Ab\in\RR^{d\times d}$ and a random Gaussian vector $\ub\sim \cN(\bm{0}, \I_d)$, and an integer $1\leq p\leq d$, then it holds that
\begin{align}
\EE\left[ \left(  \sum_{i=1}^{d} \Ab_{i,i} \ub_i^2 \right)^2 \right]
=&
\Big(\tr(\Ab)\Big)^2 + 2 \sum_{i=1}^{d} \Ab_{i,i}^2, \label{eq:Auu0}\\
\EE\left[ \left(  \sum_{i=1}^{d} \Ab_{i,i} \ub_i^2 \ub_p\right)^2 \right] \label{eq:Auu1}
=&
\Big(\tr(\Ab)\Big)^2 
+ 4 \Ab_{p,p} \tr(\Ab) 
+ 2 \sum_{i=1}^{d} \Ab_{i,i}^2 
+ 8 \Ab_{p,p}^2,\\
\EE\left[ \left(  \sum_{i=1}^{d} \Ab_{i,i} \ub_i^2 \ub_p^2\right)^2 \right] 
=& 
3 \Big(\tr(\Ab)\Big)^2 
+ 6 \sum_{i=1}^{d} \Ab_{i,i}^2 
+ 24 \Ab_{p,p} \cdot \tr(\Ab) + 72 \Ab_{p,p}^2. \label{eq:Auu_p}
\end{align}
\end{lemma}
\begin{proof}
First, we have
\begin{align*}
&\EE\left[ \left(  \sum_{i=1}^{d} \Ab_{i,i} \ub_i^2 \ub_p^n\right)^2 \right]
=
\EE\left[ \sum_{i=1}^{d}\sum_{j=1}^{d} \Ab_{i,i} \Ab_{j,j} \ub_i^2 \ub_j^2 \ub_p^{2n} \right]\\
=&
\EE\left[ \sum_{i\neq j\neq p} \Ab_{i,i} \Ab_{j,j}  \ub_i^2 \ub_j^2 \ub_p^{2n} \right] + \EE\left[ \sum_{i\neq j, i = p} \Ab_{i,i} \Ab_{j,j}  \ub_i^2 \ub_j^2 \ub_p^{2n}  \right] \\
&+ \EE\left[ \sum_{i\neq j, j = p} \Ab_{i,i} \Ab_{j,j}  \ub_i^2 \ub_j^2 \ub_p^{2n}  \right] 
+ \EE\left[ \sum_{i= j, j \neq p} \Ab_{i,i} \Ab_{j,j}  \ub_i^2 \ub_j^2 \ub_p^{2n}  \right] 
+ \EE\left[ \sum_{i= j = p} \Ab_{i,i} \Ab_{j,j}  \ub_i^2 \ub_j^2 \ub_p^{2n} \right].
\end{align*}
Next, we will bound above terms and we have
\begin{align*}
\EE\left[ \sum_{i\neq j\neq p} \Ab_{i,i} \Ab_{j,j}  \ub_i^2 \ub_j^2 \ub_p^{2n} \right] 
= \sum_{i\neq j\neq p} \Ab_{i,i} \Ab_{j,j} \EE\left[ \ub_i^2  \right] \EE\left[ \ub_j^2 \right] \EE\left[ \ub_p^{2n} \right]
=
\begin{cases}
\sum_{i\neq j\neq p} \Ab_{i,i} \Ab_{j,j}, \qquad n = 0\\
\sum_{i\neq j\neq p} \Ab_{i,i} \Ab_{j,j}, \qquad n = 1\\
3\sum_{i\neq j\neq p} \Ab_{i,i} \Ab_{j,j} \qquad n = 2
\end{cases},
\end{align*}
where the first equality is because $\ub_i$, $\ub_j$ and $\ub_p$ are independent when $i\neq j\neq p$ and the last equality is because of Eq.~\eqref{eq:Eu}. 
Similarly, we have
\begin{align*}
&\EE\left[ \sum_{i\neq j, i = p} \Ab_{i,i} \Ab_{j,j}  \ub_i^2 \ub_j^2 \ub_p^{2n}  \right]
= \EE\left[ \sum_{j\neq  p} \Ab_{p,p} \Ab_{j,j}   \ub_j^2 \ub_p^{2n + 2}  \right]
\\
=& 
\sum_{j\neq  p} \Ab_{p,p} \Ab_{j,j} \EE\left[ \ub_j^2 \right] \EE\left[ \ub_p^{2n + 2} \right]
\stackrel{\eqref{eq:Eu}}{=}
\begin{cases}
\sum_{j\neq  p} \Ab_{p,p} \Ab_{j,j} \qquad &n = 0 \\
3 \sum_{j\neq  p} \Ab_{p,p} \Ab_{j,j} \qquad &n = 1\\
15 \sum_{j\neq  p} \Ab_{p,p} \Ab_{j,j} \qquad &n = 2
\end{cases} ,
\end{align*}
and
\begin{align*}
&\EE\left[ \sum_{i\neq j, j = p} \Ab_{i,i} \Ab_{j,j}  \ub_i^2 \ub_j^2 \ub_p^{2n}  \right] 
=
\EE\left[ \sum_{i\neq p}^{d} \Ab_{i,i} \Ab_{p,p} \ub_i^2 \ub_p^{2n+2} \right]\\
=& \sum_{i\neq p}^{d} \Ab_{i,i} \Ab_{p,p} \EE\left[ \ub_i^2 \right] \EE\left[ \ub_p^{2n+2} \right]
\stackrel{\eqref{eq:Eu}}{=} 
\begin{cases}
\sum_{i\neq  p} \Ab_{p,p} \Ab_{i,i} \qquad &n = 0\\
3 \sum_{i\neq  p} \Ab_{p,p} \Ab_{i,i} \qquad &n = 1\\
15 \sum_{i\neq  p} \Ab_{p,p} \Ab_{i,i} \qquad &n = 2
\end{cases}.
\end{align*}
Furthermore, we have
\begin{align*}
&\EE\left[ \sum_{i= j, j \neq p} \Ab_{i,i} \Ab_{j,j}  \ub_i^2 \ub_j^2 \ub_p^{2n}  \right] 
= \EE\left[ \sum_{i \neq p}^{d} \Ab_{i,i}^2 \ub_i^4\ub_p^{2n} \right]\\
=& \sum_{i \neq p}^{d} \Ab_{i,i}^2 \EE\left[ \ub_i^4 \right] \EE\left[\ub_p^{2n}\right]
\stackrel{\eqref{eq:Eu}}{=} 
\begin{cases}
3 \sum_{i\neq p}^{d} \Ab_{i,i}^2\qquad &n = 0\\
3 \sum_{i\neq p}^{d} \Ab_{i,i}^2\qquad &n = 1\\
9 \sum_{i\neq p}^{d} \Ab_{i,i}^2\qquad &n = 2
\end{cases},
\end{align*}
and
\begin{align*}
\EE\left[ \sum_{i= j = p} \Ab_{i,i} \Ab_{j,j}  \ub_i^2 \ub_j^2 \ub_p^{2n} \right]
= \EE\left[ \Ab_{p,p}^2 \ub_p^{2n+4} \right] 
\stackrel{\eqref{eq:Eu}}{=} 
\begin{cases}
3 \Ab_{p,p}^2\qquad &n = 0\\
15 \Ab_{p,p}^2\qquad &n = 1\\
105 \Ab_{p,p}^2\qquad &n = 2
\end{cases}.
\end{align*}

Thus, when $n = 0$, we can obtain that
\begin{align*}
\EE\left[ \left(  \sum_{i=1}^{d} \Ab_{i,i} \ub_i^2 \right)^2 \right]
=& 
\sum_{i\neq j\neq p} \Ab_{i,i} \Ab_{j,j}
+ \sum_{j\neq  p} \Ab_{p,p} \Ab_{j,j}
+ \sum_{i\neq  p} \Ab_{p,p} \Ab_{i,i}
+ 3 \sum_{i\neq p}^{d} \Ab_{i,i}^2
+ 3 \Ab_{p,p}^2\\
=&
\sum_{i=1, j=1}^d \Ab_{i,i} \Ab_{j,j} + 2\sum_{i\neq p}^{d} \Ab_{i,i}^2 + 2\Ab_{p,p}^2\\
=&
\Big(\tr(\Ab)\Big)^2 + 2 \sum_{i=1}^{d} \Ab_{i,i}^2.
\end{align*}

When $n = 1$, we can obtain that
\begin{align*}
\EE\left[ \left(  \sum_{i=1}^{d} \Ab_{i,i} \ub_i^2 \ub_p\right)^2 \right]
=& 
\sum_{i\neq j\neq p} \Ab_{i,i} \Ab_{j,j} 
+ 3 \sum_{j\neq  p} \Ab_{p,p} \Ab_{j,j}
+ 3 \sum_{i\neq  p} \Ab_{p,p} \Ab_{i,i}
+ 3 \sum_{i\neq p}^{d} \Ab_{i,i}^2
+ 15 \Ab_{p,p}^2\\
=&
\sum_{i=1, j=1}^d \Ab_{i,i} \Ab_{j,j}
+ 2 \sum_{j\neq  p} \Ab_{p,p} \Ab_{j,j}
+ 2 \sum_{i\neq  p} \Ab_{p,p} \Ab_{i,i}
+ 2 \sum_{i\neq p}^{d} \Ab_{i,i}^2
+ 14 \Ab_{p,p}^2\\
=&
\Big(\tr(\Ab)\Big)^2 + 4 \Ab_{p,p} \tr(\Ab) + 2 \sum_{i=1}^{d} \Ab_{i,i}^2 + 8 \Ab_{p,p}^2.
\end{align*}

When $n = 2$, we obtain that
\begin{align*}
\EE\left[ \left(  \sum_{i=1}^{d} \Ab_{i,i} \ub_i^2 \ub_p^2\right)^2 \right]
=& 3 \sum_{i\neq j\neq p} \Ab_{i,i} \Ab_{j,j}  
+  15 \sum_{j\neq  p} \Ab_{p,p} \Ab_{j,j} + 15 \sum_{i\neq  p} \Ab_{p,p} \Ab_{i,i} 
+ 9 \sum_{i\neq p}^{d} \Ab_{i,i}^2 + 105 \Ab_{p,p}^2\\
=& 3 \sum_{i=1, j=1}^d \Ab_{i,i} \Ab_{j,j}  + 6 \sum_{i\neq p}^{d} \Ab_{i,i}^2 + 12  \sum_{j\neq  p} \Ab_{p,p} \Ab_{j,j} + 12 \sum_{i\neq  p} \Ab_{p,p} \Ab_{i,i} + 102 \Ab_{p,p}^2\\
=& 3 \Big(\tr(\Ab)\Big)^2 + 6 \sum_{i=1}^{d} \Ab_{i,i}^2 + 24 \Ab_{p,p} \cdot \tr(\Ab) + 72 \Ab_{p,p}^2.
\end{align*}
\end{proof}

\subsection{Bound of Second Term}

We will decompose the second term of the right hand of Eq.~\eqref{eq:uAu} in Lemma~\ref{lem:decom}. 
First, we provide the following lemma.
\begin{lemma}
\label{lem:sqr}
Given a matrix $\Ab\in\RR^{d\times d}$ and a random Gaussian vector $\ub\sim \cN(\bm{0}, \I_d)$, and an integer $1\leq p\leq d$, then it holds that
\begin{align*}
\EE\left[ \left(  \sum_{i> j}^{d} \Ab_{i,j} \ub_i\ub_j \ub_p^n \right)^2 \right] 
=&
\EE\left[  \sum_{i>j}^{d}  \Ab_{i,j}^2 \ub_i^2\ub_j^2 \ub_p^{2n} \right]\\
\EE\left[ \left(  \sum_{i< j}^{d} \Ab_{i,j} \ub_i\ub_j \ub_p^n \right)^2 \right] 
=&
\EE\left[  \sum_{i<j}^{d}  \Ab_{i,j}^2 \ub_i^2\ub_j^2 \ub_p^{2n} \right].
\end{align*}
\end{lemma}
\begin{proof}
We have
\begin{align*}
&\EE\left[ \left(  \sum_{i> j}^{d} \Ab_{i,j} \ub_i\ub_j \ub_p^n \right)^2 \right] 
=\EE\left[\sum_{i>j}^{d} \sum_{k>\ell}^{d} \Ab_{i,j} \Ab_{k,\ell}  \ub_i\ub_j \ub_k\ub_\ell \ub_p^{2n} \right] \\
=&  \EE\left[  \sum_{i>j}^{d}  \Ab_{i,j}^2 \ub_i^2\ub_j^2 \ub_p^{2n} \right] 
+\EE\left[ \sum_{i> j,\; k=i,\; \ell \neq j,\; k>\ell} \Ab_{i,j} \Ab_{k,\ell} \ub_i\ub_j \ub_k\ub_\ell \ub_p^{2n} \right] \\
&+\EE\left[ \sum_{i> j,\; k\neq i,\; \ell = j,\;k>\ell} \Ab_{i,j} \Ab_{k,\ell} \ub_i\ub_j \ub_k\ub_\ell \ub_p^{2n} \right] 
+\EE\left[ \sum_{i> j,\; k\neq i,\; \ell \neq j,\; k>\ell} \Ab_{i,j} \Ab_{k,\ell} \ub_i\ub_j \ub_k\ub_\ell \ub_p^{2n} \right]\\
&
+ \EE\left[ \sum_{i>j,\; k> \ell,\; i = \ell} \Ab_{i,j} \Ab_{k,\ell} \ub_i \ub_j\ub_k \ub_\ell \ub_p^{2n} \right]
+ \EE\left[ \sum_{i>j,\; k >\ell,\; j = k} \Ab_{i,j}\Ab_{k,\ell} \ub_i \ub_j^2 \ub_\ell \ub_p^{2n} \right].
\end{align*}
We also have 
\begin{align*}
\EE\left[ \sum_{i> j,\; k=i,\; \ell \neq j,\;k>\ell} \Ab_{i,j} \Ab_{k,\ell} \ub_i\ub_j \ub_k\ub_\ell \ub_p^{2n} \right]
=
\sum_{i> j,\; k=i\; \ell \neq j,\;k>\ell} \Ab_{i,j} \Ab_{i,\ell} \EE\left[  \ub_i^2 \ub_j \ub_\ell \ub_p^{2n} \right] 
=
0,
\end{align*}
where the last equality is because of $\ell \neq j$ and Eq.~\eqref{eq:zero}.

Similarly, we can obtain 
\begin{align*}
\EE\left[ \sum_{i> j,\; k\neq i\; \ell = j,\;k>\ell} \Ab_{i,j} \Ab_{k,\ell} \ub_i\ub_j \ub_k\ub_\ell \ub_p^{2n} \right]
=&
\sum_{i> j,\; k\neq i\; \ell = j,\;k>\ell} \Ab_{i,j} \Ab_{k,\ell} \EE\left[  \ub_i\ub_j^2 \ub_k \ub_p^{2n} \right]
= 0,\\
\EE\left[ \sum_{i>j,\; k> \ell,\; i = \ell} \Ab_{i,j} \Ab_{k,\ell} \ub_i \ub_j\ub_k \ub_\ell \ub_p^{2n} \right] 
=&
\sum_{i>j,\; k> \ell,\; i = \ell}\EE\left[  \Ab_{i,j} \Ab_{k,\ell} \ub_i^2 \ub_j\ub_k \ub_p^{2n} \right]
=0, \\
\EE\left[ \sum_{i>j,\; k >\ell,\; j = k} \Ab_{i,j}\Ab_{k,\ell} \ub_i \ub_j\ub_k \ub_\ell \ub_p^{2n} \right]
=& 
\sum_{i>j,\; k >\ell,\; j = k}\EE\left[  \Ab_{i,j}\Ab_{k,\ell} \ub_i \ub_j^2 \ub_\ell \ub_p^{2n} \right]
= 0.
\end{align*}
We also have
\begin{align*}
\EE\left[ \sum_{i> j,\; k\neq i,\; \ell \neq j,\; k>\ell} \Ab_{i,j} \Ab_{k,\ell} \ub_i\ub_j \ub_k\ub_\ell \ub_p^{2n} \right]
=
\sum_{i> j,\; k\neq i,\; \ell \neq j,\; k>\ell}
\EE\left[ \Ab_{i,j} \Ab_{k,\ell} \ub_i\ub_j \ub_k\ub_\ell \ub_p^{2n} \right]
= 
0.
\end{align*}

Thus, we can obtain that
\begin{equation*}
\EE\left[ \left(  \sum_{i> j}^{d} \Ab_{i,j} \ub_i\ub_j \ub_p^n \right)^2 \right] 
=
\EE\left[  \sum_{i>j}^{d}  \Ab_{i,j}^2 \ub_i^2\ub_j^2 \ub_p^{2n} \right]. 
\end{equation*}

Similarly, we can obtain that
\begin{equation*}
\EE\left[ \left(  \sum_{i< j}^{d} \Ab_{i,j} \ub_i\ub_j \ub_p^n \right)^2 \right] 
=
\EE\left[  \sum_{i<j}^{d}  \Ab_{i,j}^2 \ub_i^2\ub_j^2 \ub_p^{2n} \right].
\end{equation*}
\end{proof}

\begin{lemma}\label{lem:decom}
Given a matrix $\Ab\in\RR^{d\times d}$ and a random Gaussian vector $\ub\sim \cN(\bm{0}, \I_d)$, and an integer $1\leq p\leq d$, then it holds that
\begin{equation}\label{eq:Auup}
\begin{aligned}
\EE\left[ \left(  \sum_{i\neq j}^{d} \Ab_{i,j} \ub_i\ub_j \ub_p^n \right)^2 \right]
=&
\EE\left[  \sum_{i>j}^{d}  \Ab_{i,j}^2 \ub_i^2\ub_j^2 \ub_p^{2n} \right]
+ \EE\left[  \sum_{i<j}^{d}  \Ab_{i,j}^2 \ub_i^2\ub_j^2 \ub_p^{2n} \right]\\
& 
+ 2\EE\left[ \sum_{i> j}^{d} \Ab_{i,j} \ub_i\ub_j \ub_p^n \cdot \sum_{i< j}^{d} \Ab_{i,j} \ub_i\ub_j \ub_p^n  \right].	
\end{aligned}
\end{equation}
\end{lemma}
\begin{proof}
We have
\begin{align*}
&\EE\left[ \left(  \sum_{i\neq j}^{d} \Ab_{i,j} \ub_i\ub_j \ub_p^n \right)^2 \right] 
= 
\EE\left[ \left( \sum_{i> j}^{d} \Ab_{i,j} \ub_i\ub_j \ub_p^n  + \sum_{i < j}^{d} \Ab_{i,j} \ub_i\ub_j \ub_p^n \right)^2 \right]\\
=& 
\EE\left[  \left(  \sum_{i> j}^{d} \Ab_{i,j} \ub_i\ub_j \ub_p^n \right)^2  \right] 
+ \EE\left[  \left(  \sum_{i< j}^{d} \Ab_{i,j} \ub_i\ub_j \ub_p^n \right)^2  \right]\\
&+ 2\EE\left[ \sum_{i> j}^{d} \Ab_{i,j} \ub_i\ub_j \ub_p^n \cdot \sum_{i< j}^{d} \Ab_{i,j} \ub_i\ub_j \ub_p^n  \right]\\
=& \EE\left[  \sum_{i>j}^{d}  \Ab_{i,j}^2 \ub_i^2\ub_j^2 \ub_p^{2n} \right] 
+ \EE\left[  \sum_{i<j}^{d}  \Ab_{i,j}^2 \ub_i^2\ub_j^2 \ub_p^{2n} \right]
+ 2\EE\left[ \sum_{i> j}^{d} \Ab_{i,j} \ub_i\ub_j \ub_p^n \cdot \sum_{i< j}^{d} \Ab_{i,j} \ub_i\ub_j \ub_p^n  \right], 
\end{align*}
where the last equality is because of Lemma~\ref{lem:sqr}.
\end{proof}

Next, we will bound three terms in the right hand of Eq.~\eqref{eq:Auup}.
\begin{lemma}
\label{lem:A_ij}
Given a matrix $\Ab\in\RR^{d\times d}$ and a random Gaussian vector $\ub\sim \cN(\bm{0}, \I_d)$, and an integer $1\leq p\leq d$, then it holds that
\begin{align}
\EE\left[  \sum_{i>j}^{d}  \Ab_{i,j}^2 \ub_i^2\ub_j^2 \ub_p^{2n} \right]
=& 
\begin{cases}
\sum_{i>j}^{d} \Ab_{i,j}^2  \qquad &n=0 \\
\sum_{i>j}^{d} \Ab_{i,j}^2
+ 2 \sum_{i>j,  j= p}^{d} \Ab_{i,p}^2 
+ 2 \sum_{i>j,  i= p}^{d} \Ab_{p,j}^2 \qquad & n = 1\\
3 \sum_{i>j}^{d} \Ab_{i,j}^2 
+ 12 \sum_{i>j,  j= p}^{d} \Ab_{i,p}^2 
+ 12 \sum_{i>j,  i= p}^{d} \Ab_{p,j}^2 \qquad &n=2
\end{cases}
\label{eq:A_ij}\\
\EE\left[  \sum_{i<j}^{d}  \Ab_{i,j}^2 \ub_i^2\ub_j^2 \ub_p^{2n} \right] 
=& 
\begin{cases}
\sum_{i<j}^{d} \Ab_{i,j}^2  \qquad &n=0 \\
\sum_{i<j}^{d} \Ab_{i,j}^2
+ 2 \sum_{i<j,  j= p}^{d} \Ab_{i,p}^2 
+ 2 \sum_{i<j,  i= p}^{d} \Ab_{p,j}^2 \qquad & n = 1\\
3 \sum_{i<j}^{d} \Ab_{i,j}^2 
+ 12 \sum_{i<j,  j= p}^{d} \Ab_{i,p}^2 
+ 12 \sum_{i<j,  i= p}^{d} \Ab_{p,j}^2 \qquad & n=2
\end{cases}.
\end{align}
\end{lemma}
\begin{proof}
First, we have
\begin{equation}\label{eq:EE}
\begin{aligned}
&\EE\left[  \sum_{i>j}^{d}  \Ab_{i,j}^2 \ub_i^2\ub_j^2 \ub_p^{2n} \right] \\
=& 
\EE\left[  \sum_{i>j, i\neq p, j\neq p}^{d} \Ab_{i,j}^2 \ub_i^2\ub_j^2 \ub_p^{2n} \right]
+
\EE\left[  \sum_{i>j,  j= p}^{d} \Ab_{i,j}^2 \ub_i^2\ub_j^2 \ub_p^{2n} \right] 
+ \EE\left[  \sum_{i>j,  i= p}^{d} \Ab_{i,j}^2 \ub_i^2\ub_j^2 \ub_p^{2n} \right]\\
=&
\sum_{i>j, i\neq p, j\neq p}^{d} \Ab_{i,j}^2 \EE\left[\ub_i^2\right] \EE\left[\ub_j^2\right] \EE\left[\ub_p^{2n}\right]
+ \sum_{i>p,  j= p}^{d} \Ab_{i,p}^2 \EE\left[   \ub_i^2  \right] \EE\left[\ub_p^{2n+2}\right] 
+ \sum_{p>j,  i= p}^{d} \Ab_{p,j}^2 \EE\left[   \ub_j^2  \right] \EE\left[\ub_p^{2n+2}\right].
\end{aligned}	
\end{equation}
When $n = 0$, by Eq.~\eqref{eq:Eu}, Eq.~\eqref{eq:EE} reduces to
\begin{align*}
\sum_{i>j, i\neq p, j\neq p}^{d} \Ab_{i,j}^2
+\sum_{i>j,  j= p}^{d} \Ab_{i,p}^2 
+ \sum_{i>j,  i= p}^{d} \Ab_{p,j}^2
=
\sum_{i>j}^{d} \Ab_{i,j}^2.
\end{align*}

When $n = 1$, by Eq.~\eqref{eq:Eu}, Eq.~\eqref{eq:EE} reduces to
\begin{align*}
\sum_{i>j, i\neq p, j\neq p}^{d} \Ab_{i,j}^2
+3\sum_{i>j,  j= p}^{d} \Ab_{i,p}^2 
+ 3\sum_{i>j,  i= p}^{d} \Ab_{p,j}^2
=
\sum_{i>j}^{d} \Ab_{i,j}^2
+ 2 \sum_{i>j,  j= p}^{d} \Ab_{i,p}^2 
+ 2 \sum_{i>j,  i= p}^{d} \Ab_{p,j}^2.
\end{align*}

When $n = 2$, by Eq.~\eqref{eq:Eu}, Eq.~\eqref{eq:EE} reduces to 
\begin{align*}
&3 \sum_{i>j, i\neq p, j\neq p}^{d} \Ab_{i,j}^2  
+ 15 \sum_{i>j,  j= p}^{d} \Ab_{i,p}^2 
+ 15 \sum_{i>j,  i= p}^{d} \Ab_{p,j}^2 \\
=& 
3 \sum_{i>j}^{d} \Ab_{i,j}^2 
+ 12 \sum_{i>j,  j= p}^{d} \Ab_{i,p}^2 
+ 12 \sum_{i>j,  i= p}^{d} \Ab_{p,j}^2.
\end{align*}

Combining above results, we can obtain
\begin{equation*}
\EE\left[  \sum_{i>j}^{d}  \Ab_{i,j}^2 \ub_i^2\ub_j^2 \ub_p^{2n} \right]
= 
\begin{cases}
\sum_{i>j}^{d} \Ab_{i,j}^2  \qquad &n=0 \\
\sum_{i>j}^{d} \Ab_{i,j}^2
+ 2 \sum_{i>j,  j= p}^{d} \Ab_{i,p}^2 
+ 2 \sum_{i>j,  i= p}^{d} \Ab_{p,j}^2 \qquad & n = 1\\
3 \sum_{i>j}^{d} \Ab_{i,j}^2 
+ 12 \sum_{i>j,  j= p}^{d} \Ab_{i,p}^2 
+ 12 \sum_{i>j,  i= p}^{d} \Ab_{p,j}^2 \qquad &n=2
\end{cases}.
\end{equation*}

Similarly, we can obtain that
\begin{align*}
\EE\left[  \sum_{i>j}^{d}  \Ab_{i,j}^2 \ub_i^2\ub_j^2 \ub_p^{2n} \right] 
= 
\begin{cases}
\sum_{i<j}^{d} \Ab_{i,j}^2  \qquad &n=0 \\
\sum_{i<j}^{d} \Ab_{i,j}^2
+ 2 \sum_{i<j,  j= p}^{d} \Ab_{i,p}^2 
+ 2 \sum_{i<j,  i= p}^{d} \Ab_{p,j}^2 \qquad & n = 1\\
3 \sum_{i<j}^{d} \Ab_{i,j}^2 
+ 12 \sum_{i<j,  j= p}^{d} \Ab_{i,p}^2 
+ 12 \sum_{i<j,  i= p}^{d} \Ab_{p,j}^2 \qquad & n=2
\end{cases}.
\end{align*}
\end{proof}

\begin{lemma}
\label{lem:cross}
Given a matrix $\Ab\in\RR^{d\times d}$ and a random Gaussian vector $\ub\sim \cN(\bm{0}, \I_d)$, and an integer $1\leq p\leq d$, then it holds that
\begin{align*}
&\EE\left[ \sum_{i> j}^{d} \Ab_{i,j} \ub_i\ub_j \ub_p^n \cdot \sum_{k< \ell}^{d} \Ab_{k,\ell} \ub_k\ub_\ell \ub_p^n  \right] \\
=&
\begin{cases}
\sum_{i>j}^{d} \Ab_{i,j} \Ab_{j,i} \qquad &n=0\\
\sum_{i>j}^{d} \Ab_{i,j} \Ab_{j,i} 
+ 2 \sum_{i>j,  j= p}^{d} \Ab_{i,p} \Ab_{p,i} 
+ 2 \sum_{i>j,  i= p}^{d} \Ab_{p,j} \Ab{j,p}  \qquad & n=1 \\
3 \sum_{i>j}^{d} \Ab_{i,j} \Ab_{j,i} 
+ 12 \sum_{i>j,  j= p}^{d} \Ab_{i,p} \Ab_{p,i} 
+ 12 \sum_{i>j,  i= p}^{d} \Ab_{p,j} \Ab{j,p} \qquad &n=2
\end{cases}.
\end{align*}
\end{lemma}
\begin{proof}
First, we have
\begin{align*}
&\EE\left[ \sum_{i> j}^{d} \Ab_{i,j} \ub_i\ub_j \ub_p^2 \cdot \sum_{k< \ell}^{d} \Ab_{k,\ell} \ub_k\ub_\ell \ub_p^2  \right]\\
=&
\EE\left[ \sum_{i>j}^{d} \sum_{k< \ell,\; k =j,\; \ell = i}^{d} \Ab_{i,j}  \Ab_{k,\ell} \ub_i  \ub_j    \ub_k\ub_\ell \ub_p^{2n}\right]
+ 
\EE\left[ \sum_{i>j}^{d} \sum_{k< \ell,\; k \neq j,\; \ell = i}^{d} \Ab_{i,j}  \Ab_{k,\ell} \ub_i  \ub_j    \ub_k\ub_\ell \ub_p^{2n} \right]\\
&
+\EE\left[ \sum_{i>j}^{d} \sum_{k< \ell,\; k = j,\; \ell \neq i}^{d} \Ab_{i,j}  \Ab_{k,\ell} \ub_i  \ub_j    \ub_k\ub_\ell \ub_p^{2n} \right] 
+
\EE\left[ \sum_{i>j}^{d} \sum_{k< \ell,\; k \neq j,\; \ell \neq i}^{d} \Ab_{i,j}  \Ab_{k,\ell} \ub_i  \ub_j    \ub_k\ub_\ell \ub_p^{2n} \right]\\
&
+\EE\left[ \sum_{i>j}^{d} \sum_{k< \ell,\; j = \ell}^{d} \Ab_{i,j}  \Ab_{k,\ell} \ub_i  \ub_j    \ub_k\ub_\ell \ub_p^{2n} \right]
+
\EE\left[ \sum_{i>j}^{d} \sum_{k< \ell,\; k = i}^{d} \Ab_{i,j}  \Ab_{k,\ell} \ub_i  \ub_j    \ub_k\ub_\ell \ub_p^{2n} \right].
\end{align*}
Then, we have
\begin{align*}
\EE\left[ \sum_{i>j}^{d} \sum_{k< \ell,\; k \neq j,\; \ell = i}^{d} \Ab_{i,j}  \Ab_{k,\ell} \ub_i  \ub_j    \ub_k\ub_\ell \ub_p^{2n} \right]
= \sum_{i>j}^{d} \sum_{k< \ell,\; k \neq j,\; \ell = i}^{d} \Ab_{i,j}  \Ab_{k,\ell} \EE\left[ \ub_i^2 \ub_j    \ub_k \ub_p^{2n} \right]
= 0,
\end{align*}
where the last equality is because of Eq.~\eqref{eq:zero}.

Similarly, we can obtain that
\begin{align*}
&\EE\left[ \sum_{i>j}^{d} \sum_{k< \ell,\; k = j,\; \ell \neq i}^{d} \Ab_{i,j}  \Ab_{k,\ell} \ub_i  \ub_j    \ub_k\ub_\ell \ub_p^{2n} \right]
=
\sum_{i>j}^{d} \sum_{k< \ell,\; k = j,\; \ell \neq i}^{d} \Ab_{i,j}  \Ab_{k,\ell} \EE\left[\ub_i\ub_j^2 \ub_\ell \ub_p^{2n}\right] =0,\\
&
\EE\left[ \sum_{i>j}^{d} \sum_{k< \ell,\; k \neq j,\; \ell \neq i}^{d} \Ab_{i,j}  \Ab_{k,\ell} \ub_i  \ub_j    \ub_k\ub_\ell \ub_p^{2n} \right]
= \sum_{i>j}^{d} \sum_{k< \ell,\; k \neq j,\; \ell \neq i}^{d} \Ab_{i,j}  \Ab_{k,\ell} \EE\left[ \ub_i  \ub_j    \ub_k\ub_\ell \ub_p^{2n} \right] = 0,\\
&
\EE\left[ \sum_{i>j}^{d} \sum_{k< \ell,\; j = \ell}^{d} \Ab_{i,j}  \Ab_{k,\ell} \ub_i  \ub_j    \ub_k\ub_\ell \ub_p^{2n} \right]
=\sum_{i>j}^{d} \sum_{k< \ell,\; j = \ell}^{d} \Ab_{i,j}  \Ab_{k,\ell}
\EE\left[\ub_i \ub_j^2 \ub_k \ub_p^{2n}\right] =0,\\
&
\EE\left[ \sum_{i>j}^{d} \sum_{k< \ell,\; k = i}^{d} \Ab_{i,j}  \Ab_{k,\ell} \ub_i  \ub_j    \ub_k\ub_\ell \ub_p^{2n} \right]
=\sum_{i>j}^{d} \sum_{k< \ell,\; k = i}^{d} \Ab_{i,j}  \Ab_{k,\ell} \EE\left[ \ub_i^2 \ub_j \ub_\ell \ub_p^{2n} \right] =0.
\end{align*}

Thus, we can obtain that
\begin{align*}
&\EE\left[ \sum_{i> j}^{d} \Ab_{i,j} \ub_i\ub_j \ub_p^n \cdot \sum_{k< \ell}^{d} \Ab_{k,\ell} \ub_k\ub_\ell \ub_p^n  \right]\\
=&
\EE\left[ \sum_{i>j}^{d} \sum_{k< \ell,\; k =j,\; \ell = i}^{d} \Ab_{i,j}  \Ab_{k,\ell} \ub_i  \ub_j    \ub_k\ub_\ell \ub_p^{2n}\right]\\
=& \EE\left[\sum_{i>j}^{d} \Ab_{i,j} \Ab_{j,i} \ub_i^2\ub_j^2 \ub_p^{2n}\right].
\end{align*}

Define a matrix $\Bb\in\RR^{d\times d}$ with $\Bb_{i,j} = \Ab_{i,j} \Ab_{j,i}$.
Then, by Eq.~\eqref{eq:A_ij}, we can obtain that
\begin{align*}
\EE\left[  \sum_{i>j}^{d}  \Bb_{i,j} \ub_i^2\ub_j^2  \right]
=& 
\sum_{i>j}^{d} \Ab_{i,j} \Ab_{j,i},\\
\EE\left[  \sum_{i>j}^{d}  \Bb_{i,j} \ub_i^2\ub_j^2 \ub_p^2 \right]
=&
\sum_{i>j}^{d} \Bb_{i,j} 
+ 2 \sum_{i>j,  j= p}^{d} \Bb_{i,p} 
+ 2 \sum_{i>j,  i= p}^{d} \Bb_{p,j}\\
=&
\sum_{i>j}^{d} \Ab_{i,j} \Ab_{j,i} 
+ 2 \sum_{i>j,  j= p}^{d} \Ab_{i,p} \Ab_{p,i} 
+ 2 \sum_{i>j,  i= p}^{d} \Ab_{p,j} \Ab{j,p}.
\end{align*}
and,
\begin{align*}
\EE\left[  \sum_{i>j}^{d}  \Bb_{i,j} \ub_i^2\ub_j^2 \ub_p^4 \right]
=& 
3 \sum_{i>j}^{d} \Bb_{i,j} 
+ 12 \sum_{i>j,  j= p}^{d} \Bb_{i,p} 
+ 12 \sum_{i>j,  i= p}^{d} \Bb_{p,j}\\
=& 
3 \sum_{i>j}^{d} \Ab_{i,j} \Ab_{j,i} 
+ 12 \sum_{i>j,  j= p}^{d} \Ab_{i,p} \Ab_{p,i} 
+ 12 \sum_{i>j,  i= p}^{d} \Ab_{p,j} \Ab{j,p}.
\end{align*}
\end{proof}

Based on above lemmas, we will provide the explicit bounds of $\EE\left[ \left(  \sum_{i\neq j}^{d} \Ab_{i,j} \ub_i\ub_j \ub_p^n \right)^2 \right]$ with $n = 0,1,2$ in the next lemmas.

\begin{lemma}
Given a matrix $\Ab\in\RR^{d\times d}$ and a random Gaussian vector $\ub\sim \cN(\bm{0}, \I_d)$, and an integer $1\leq p\leq d$, then it holds that
	\begin{equation}\label{eq:Auu2}
		\EE\left[ \left(  \sum_{i\neq j}^{d} \Ab_{i,j} \ub_i\ub_j  \right)^2 \right]
		= 
		\sum_{i>j}^{d} \Big(\Ab_{i,j} + \Ab_{j,i}\Big)^2.
	\end{equation}
\end{lemma}
\begin{proof}
	By Eq.~\eqref{eq:Auup} with $n = 0$, we can obtain that
	\begin{align*}
		\EE\left[ \left(  \sum_{i\neq j}^{d} \Ab_{i,j} \ub_i\ub_j  \right)^2 \right]
		=
		\EE\left[  \sum_{i>j}^{d}  \Ab_{i,j}^2 \ub_i^2\ub_j^2  \right]
		+ \EE\left[  \sum_{i<j}^{d}  \Ab_{i,j}^2 \ub_i^2\ub_j^2  \right]
		+ 2\EE\left[ \sum_{i> j}^{d} \Ab_{i,j} \ub_i\ub_j  \cdot \sum_{i< j}^{d} \Ab_{i,j} \ub_i\ub_j   \right].	
	\end{align*}
	Combining with Lemma~\ref{lem:A_ij} and Lemma~\ref{lem:cross} wit $n=0$,
	we can obtain that
	\begin{align*}
		\EE\left[ \left(  \sum_{i\neq j}^{d} \Ab_{i,j} \ub_i\ub_j  \right)^2 \right]
		= 
		\sum_{i>j}^{d}\Ab_{i,j}^2
		+ \sum_{i<j}^{d}\Ab_{i,j}^2 
		+ 2 \sum_{i>j}^{d} \Ab_{i,j} \Ab_{j,i}
		=
		\sum_{i>j}^{d} \Big(\Ab_{i,j} + \Ab_{j,i}\Big)^2.
	\end{align*}
\end{proof}

\begin{lemma}
Given a matrix $\Ab\in\RR^{d\times d}$ and a random Gaussian vector $\ub\sim \cN(\bm{0}, \I_d)$, and an integer $1\leq p\leq d$, then it holds that
	\begin{equation}\label{eq:Auu3}
		\begin{aligned}
			\EE\left[ \left(  \sum_{i\neq j}^{d} \Ab_{i,j} \ub_i\ub_j \ub_p \right)^2 \right]
			=& 
			\sum_{i>j}^{d} \Big(\Ab_{i,j} + \Ab_{j,i}\Big)^2
			+ 2 \sum_{i>j,  j= p}^{d}\Big(\Ab_{i,p} + \Ab_{p,i} \Big)^2\\
			&
			+ 2 \sum_{i>j,  i= p}^{d}\Big( \Ab_{p,j} + \Ab_{j,p} \Big)^2.
		\end{aligned}	
	\end{equation}
\end{lemma}
\begin{proof}
	By Eq.~\eqref{eq:Auup} with $n = 1$, we can obtain that
	\begin{align*}
		\EE\left[ \left(  \sum_{i\neq j}^{d} \Ab_{i,j} \ub_i\ub_j \ub_p \right)^2 \right]  
		\stackrel{\eqref{eq:Auup}}{=}& \EE\left[  \sum_{i>j}^{d}  \Ab_{i,j}^2 \ub_i^2\ub_j^2 \ub_p^2 \right] 
		+ \EE\left[  \sum_{i<j}^{d}  \Ab_{i,j}^2 \ub_i^2\ub_j^2 \ub_p^2 \right]\\
		&
		+ 2\EE\left[ \sum_{i> j}^{d} \Ab_{i,j} \ub_i\ub_j \ub_p \cdot \sum_{i< j}^{d} \Ab_{i,j} \ub_i\ub_j \ub_p  \right]. 
	\end{align*}
	Combining with Lemma~\ref{lem:A_ij} and Lemma~\ref{lem:cross} wit $n=1$,
	we can obtain that
	\begin{align*}
		&\EE\left[ \left(  \sum_{i\neq j}^{d} \Ab_{i,j} \ub_i\ub_j \ub_p \right)^2 \right]\\  
		=&	
		\sum_{i>j}^{d} \Ab_{i,j}^2
		+ 2 \sum_{i>j,  j= p}^{d} \Ab_{i,p}^2 
		+ 2 \sum_{i>j,  i= p}^{d} \Ab_{p,j}^2\\
		&
		+ \sum_{i<j}^{d} \Ab_{i,j}^2
		+ 2 \sum_{i<j,  j= p}^{d} \Ab_{i,p}^2 
		+ 2 \sum_{i<j,  i= p}^{d} \Ab_{p,j}^2\\
		&
		+2\sum_{i>j}^{d} \Ab_{i,j} \Ab_{j,i} 
		+ 4 \sum_{i>j,  j= p}^{d} \Ab_{i,p} \Ab_{p,i} 
		+ 4 \sum_{i>j,  i= p}^{d} \Ab_{p,j} \Ab{j,p}\\
		=&
		\sum_{i>j}^{d} \Big(\Ab_{i,j}^2 + \Ab_{j,i}^2 + 2\Ab_{i,j}\Ab_{j,i}\Big)
		+ 2 \sum_{i>j,  j= p}^{d} \Big( \Ab_{i,p}^2 + \Ab_{p,i}^2 + 2\Ab_{i,p} \Ab_{p,i}  \Big)\\
		&
		+ 2 \sum_{i>j,  i= p}^{d} \Big( \Ab_{p,j}^2 + \Ab_{j,p}^2 + 2\Ab_{j,p} \Ab_{p,j}  \Big)\\
		=&
		\sum_{i>j}^{d} \Big(\Ab_{i,j} + \Ab_{j,i}\Big)^2
		+ 2 \sum_{i>j,  j= p}^{d}\Big(\Ab_{i,p} + \Ab_{p,i} \Big)^2
		+ 2 \sum_{i>j,  i= p}^{d}\Big( \Ab_{p,j} + \Ab_{j,p} \Big)^2.
	\end{align*}
\end{proof}

\begin{lemma}
Given a matrix $\Ab\in\RR^{d\times d}$ and a random Gaussian vector $\ub\sim \cN(\bm{0}, \I_d)$, and an integer $1\leq p\leq d$, then it holds that
	\begin{equation}\label{eq:Auu}
		\EE\left[ \left(  \sum_{i\neq j}^{d} \Ab_{i,j} \ub_i\ub_j \ub_p^2 \right)^2 \right] 
		= 3 \sum_{i>j}^{d} \Big(\Ab_{i,j} + \Ab_{j,i}\Big)^2
		+ 12 \sum_{i>j,  j= p}^{d} \Big(\Ab_{i,p} + \Ab_{p,i}\Big)^2
		+ 12 \sum_{i>j, i=p}^{d} \Big(\Ab_{p,j} + \Ab_{j, p}\Big)^2.
	\end{equation}
\end{lemma}
\begin{proof}
	First, by Eq.~\eqref{eq:Auup} with $n = 2$, we have
	\begin{align*}
		\EE\left[ \left(  \sum_{i\neq j}^{d} \Ab_{i,j} \ub_i\ub_j \ub_p^2 \right)^2 \right]  
		\stackrel{\eqref{eq:Auup}}{=}& \EE\left[  \sum_{i>j}^{d}  \Ab_{i,j}^2 \ub_i^2\ub_j^2 \ub_p^4 \right] 
		+ \EE\left[  \sum_{i<j}^{d}  \Ab_{i,j}^2 \ub_i^2\ub_j^2 \ub_p^4 \right]\\
		&
		+ 2\EE\left[ \sum_{i> j}^{d} \Ab_{i,j} \ub_i\ub_j \ub_p^2 \cdot \sum_{i< j}^{d} \Ab_{i,j} \ub_i\ub_j \ub_p^2  \right]. 
	\end{align*}
	Then, combining with Lemma~\ref{lem:A_ij} and Lemma~\ref{lem:cross} wit $n=2$, we can obtain that
	\begin{align*}
		\EE\left[ \left(  \sum_{i\neq j}^{d} \Ab_{i,j} \ub_i\ub_j \ub_p^2 \right)^2 \right]
		=& 3 \left( \sum_{i>j}^{d} \Ab_{i,j}^2 + \sum_{i<j}^{d} \Ab_{i,j}^2 +  2\sum_{i>j}^{d} \Ab_{i,j} \Ab_{j,i} \right)\\
		&+ 12 \left(\sum_{i>j,  j= p}^{d} \Ab_{i,p}^2  +  \sum_{i<j,  i= p}^{d} \Ab_{p,j}^2 + 2 \sum_{i>j,  j= p}^{d} \Ab_{i,p} \Ab_{p,i} \right)\\
		&+12 \left( \sum_{i>j,  i= p}^{d} \Ab_{p,j}^2 + \sum_{i<j,  j= p}^{d} \Ab_{i,p}^2 +  2\sum_{i>j,  i= p}^{d}\Ab_{p,j} \Ab{j,p} \right).
	\end{align*}
	
	Furthermore, it holds that
	\begin{align*}
		\sum_{i>j}^{d} \Ab_{i,j}^2 + \sum_{i<j}^{d} \Ab_{i,j}^2 +  2\sum_{i>j}^{d} \Ab_{i,j} \Ab_{j,i} 
		= \sum_{i>j}^{d} \left( \Ab_{i,j}^2 +  \Ab_{j,i}^2 + 2 \Ab_{i,j} \Ab_{j,i}  \right)
		= \sum_{i>j}^{d} \Big(\Ab_{i,j} + \Ab_{j,i}\Big)^2.
	\end{align*}
	Similarly, we have
	\begin{align*}
		&\sum_{i>j,  j= p}^{d} \Ab_{i,p}^2  +  \sum_{i<j,  i= p}^{d} \Ab_{p,j}^2 + 2 \sum_{i>j,  j= p}^{d} \Ab_{i,p} \Ab_{p,i}\\
		=& 
		\sum_{i>j,  j= p}^{d} \left(\Ab_{i,p}^2 + \Ab_{p,i}^2 + 2\Ab_{i,p} \Ab_{p,i}\right)
		= \sum_{i>j,  j= p}^{d} \Big(\Ab_{i,p} + \Ab_{p,i}\Big)^2,
	\end{align*}
	and
	\begin{align*}
		&\sum_{i>j,  i= p}^{d} \Ab_{p,j}^2 + \sum_{i<j,  j= p}^{d} \Ab_{i,p}^2 +  2\sum_{i>j,  i= p}^{d}\Ab_{p,j} \Ab{j,p}\\
		=& 
		\sum_{i>j, i=p}^{d} \left(\Ab_{p,j}^2 + \Ab_{j, p}^2 + 2\Ab_{p,j} \Ab{j,p}\right)
		= 
		\sum_{i>j, i=p}^{d} \Big(\Ab_{p,j} + \Ab_{j, p}\Big)^2.
	\end{align*}
	
	Thus, we can obtain that
	\begin{align*}
		\EE\left[ \left(  \sum_{i\neq j}^{d} \Ab_{i,j} \ub_i\ub_j \ub_p^2 \right)^2 \right] 
		= 3 \sum_{i>j}^{d} \Big(\Ab_{i,j} + \Ab_{j,i}\Big)^2
		+ 12 \sum_{i>j,  j= p}^{d} \Big(\Ab_{i,p} + \Ab_{p,i}\Big)^2
		+ 12 \sum_{i>j, i=p}^{d} \Big(\Ab_{p,j} + \Ab_{j, p}\Big)^2.
	\end{align*}
	
\end{proof}

\subsection{Proof of Lemma~\ref{lem:3a}}
\begin{proof}
For Eq.~\eqref{eq:a}, we have
\begin{align*}
	\EE\left[ \left(\ub^\top \Ab \ub   \right)^2 \right]
	\stackrel{\eqref{eq:uAu}}{=}&
	\EE\left[ \left(  \sum_{i=1}^{d} \Ab_{i,i} \ub_i^2 \right)^2 \right] 
	+ \EE\left[ \left(  \sum_{i\neq j}^{d} \Ab_{i,j} \ub_i\ub_j  \right)^2 \right]\\
	\stackrel{\eqref{eq:Auu0}\eqref{eq:Auu2}}{=}&
	\Big(\tr(\Ab)\Big)^2 + 2 \sum_{i=1}^{d} \Ab_{i,i}^2 
	+ \sum_{i>j}^{d} \Big(\Ab_{i,j} + \Ab_{j,i}\Big)^2.
\end{align*}

For Eq.~\eqref{eq:a1}, we have
\begin{align*}
	&\EE\left[\left(\sum_{i=1,j=1}^{d} \Ab_{i,j} \ub_i \ub_j \ub_p\right)^2\right]\\
	\stackrel{\eqref{eq:uAu}}{=}&
	\EE\left[ \left( \sum_{i\neq j}^{d} \Ab_{i,j} \ub_i\ub_j \ub_p\right)^2 \right] 
	+ \EE\left[ \left(\sum_{i=1}^{d} \Ab_{i,i} \ub_i^2 \ub_p\right)^2 \right]\\
	\stackrel{\eqref{eq:Auu1}\eqref{eq:Auu3}}{=}&
	\Big(\tr(\Ab)\Big)^2 
	+ 4 \Ab_{p,p} \tr(\Ab) 
	+ 2 \sum_{i=1}^{d} \Ab_{i,i}^2 
	+ 8 \Ab_{p,p}^2 \\
	&
	+ \sum_{i>j}^{d} \Big(\Ab_{i,j} + \Ab_{j,i}\Big)^2
	+ 2 \sum_{i>j,  j= p}^{d}\Big(\Ab_{i,p} + \Ab_{p,i} \Big)^2
	+ 2 \sum_{i>j,  i= p}^{d}\Big( \Ab_{p,j} + \Ab_{j,p} \Big)^2.
\end{align*}

For Eq.~\eqref{eq:a2}, we have
\begin{align*}
	&\EE\left[ \left(\ub^\top \Ab \ub\cdot \bm{u}_p^2  \right)^2 \right]\\
	\stackrel{\eqref{eq:uAu}}{=}&
	\EE\left[ \left(  \sum_{i=1}^{d} \Ab_{i,i} \ub_i^2 \ub_p^2\right)^2 \right] 
	+ \EE\left[ \left(  \sum_{i\neq j}^{d} \Ab_{i,j} \ub_i\ub_j \ub_p^2 \right)^2 \right]\\
	\stackrel{\eqref{eq:Auu_p}\eqref{eq:Auu}}{=}& 
	3 \Big(\tr(\Ab)\Big)^2 + 6 \sum_{i= 1}^{d} \Ab_{i,i}^2 + 24 \Ab_{p,p} \cdot \tr(\Ab) + 72 \Ab_{p,p}^2  \\
	&+ 3 \sum_{i>j}^{d} \Big(\Ab_{i,j} + \Ab_{j,i}\Big)^2
	+ 12 \sum_{i>j,  j= p}^{d} \Big(\Ab_{i,p} + \Ab_{p,i}\Big)^2
	+ 12 \sum_{i>j, i=p}^{d} \Big(\Ab_{p,j} + \Ab_{j, p}\Big)^2.
\end{align*}
\end{proof}

\end{document}